\documentclass{amsart}
\usepackage{hyperref}
\usepackage{graphicx}
\usepackage{amsmath}
\usepackage{amsthm}
\usepackage{amssymb,bbm}%
\usepackage[numbers, square]{natbib}
\usepackage{mathrsfs}
\usepackage{color}

\newcommand{\lspan}{\mathop{\mathrm{lin}}\nolimits}
\newcommand{\relint}{\mathop{\mathrm{relint}}\nolimits}

\newcommand{\rank}{\mathop{\mathrm{rank}}\nolimits}
\newcommand{\codim}{\mathop{\mathrm{codim}}\nolimits}
\renewcommand{\S}{\mathbb{S}}
\newcommand{\dist}{\mathop{\mathrm{dist}}\nolimits}

\newcommand{\cA}{\mathcal{A}}
\newcommand{\cB}{\mathcal{B}}
\newcommand{\cC}{\mathcal{C}}

\newcommand{\cR}{\mathcal{R}}

\newcommand{\Sym}{\text{\rm Sym}}

\DeclareMathOperator*{\Ker}{Ker}

\newcommand{\E}{\mathbb E}

\newcommand{\R}{\mathbb{R}}
\newcommand{\N}{\mathbb{N}}

\newcommand{\C}{\mathbb{C}}
\newcommand{\Z}{\mathbb{Z}}

\renewcommand{\P}{\mathbb{P}}

\newcommand{\simn}{\underset{n \to \infty}{\sim}}

\newcommand{\sgn}{\mathop{\mathrm{sgn}}\nolimits}
\newcommand{\conv}{\mathop{\mathrm{Conv}}\nolimits}
\newcommand{\pos}{\mathop{\mathrm{pos}}\nolimits}

\newcommand{\eps}{\varepsilon}
\newcommand{\eqdistr}{\stackrel{d}{=}}

\newcommand{\todistr}{\overset{d}{\underset{n\to\infty}\longrightarrow}}

\newcommand{\bsl}{\backslash}
\newcommand{\ind}{\mathbbm{1}}

\newcommand{\dd}{{\rm d}}
\newcommand{\eee}{{\rm e}}

\theoremstyle{plain}
\newtheorem{theorem}{Theorem}[section]
\newtheorem{lemma}[theorem]{Lemma}
\newtheorem{corollary}[theorem]{Corollary}
\newtheorem{proposition}[theorem]{Proposition}

\theoremstyle{definition}

\newtheorem{example}[theorem]{Example}
\newtheorem{problem}[theorem]{Problem}
\newtheorem{conjecture}[theorem]{Conjecture}

\theoremstyle{remark}
\newtheorem{remark}[theorem]{Remark}

\newcommand{\stirling}[2]{\genfrac{[}{]}{0pt}{}{#1}{#2}}

\begin{document}

\author{Zakhar Kabluchko}
\address{Zakhar Kabluchko, Institut f\"ur Mathematische Stochastik,
Universit\"at M\"unster,
Orl\'eans--Ring 10,
48149 M\"unster, Germany}
\email{zakhar.kabluchko@uni-muenster.de}

\author{Vladislav Vysotsky}
\address{University of Sussex and St.\ Petersburg Department of Steklov Mathematical Institute}
\email{v.vysotskiy@sussex.ac.uk, vysotsky@pdmi.ras.ru}

\author{Dmitry Zaporozhets}
\address{St.\ Petersburg Department of Steklov Mathematical Institute,
Fontanka~27,
191011 St.\ Petersburg,
Russia}
\email{zap1979@gmail.com}

\title[Convex hulls of random walks]{Convex hulls of random walks, hyperplane arrangements, and Weyl chambers}

\thanks{This paper was written when V.V. was affiliated to Imperial College London, where his work was supported by People Programme (Marie Curie Actions) of the European Union's Seventh Framework Programme (FP7/2007-2013) under REA grant agreement n$^\circ$[628803]. V.V. and D.Z. were supported in part by the RFBI Grant 16-01-00367 and by the Program of Fundamental Researches of Russian Academy of Sciences ``Modern Problems of Fundamental Mathematics''}

\begin{abstract}
We give an explicit formula for the probability that the convex hull of an $n$-step random walk in $\R^d$ does not contain the origin, under the assumption that the distribution of increments of the walk is centrally symmetric and puts no mass on affine hyperplanes. This extends the formula by Sparre Andersen (1949) for the probability that such random walk in dimension one stays positive. Our result is distribution-free, that is, the probability does not depend on the distribution of increments.

This probabilistic problem is shown to be equivalent to either of the two geometric ones:
1) Find the number of Weyl chambers of type $B_n$ intersected by a generic linear subspace of $\R^n$ of codimension $d$; 2) Find the conic intrinsic volumes of a Weyl chamber of type $B_n$. We solve the first geometric problem using the theory of hyperplane arrangements. A by-product of our method is a new simple proof of the general formula by Klivans and Swartz (2011) relating the coefficients of the characteristic polynomial of a linear hyperplane arrangement to the conic intrinsic volumes of the chambers constituting its complement.

We obtain 
analogous distribution-free results for Weyl chambers of type $A_{n-1}$ (yielding the probability of absorption of the origin by the convex hull of a generic random walk bridge), type $D_n$, and direct products of Weyl chambers (yielding the absorption probability for the joint convex hull of several random walks or bridges). The simplest case of products of the form $B_1\times \dots \times B_1$ recovers the Wendel formula (1962) for the probability that the convex hull of an i.i.d.\ multidimensional sample chosen from a centrally symmetric distribution does not contain the origin.


We also give an asymptotic analysis of the obtained absorption probabilities as $n \to \infty$, in both cases of fixed and increasing dimension~$d$.
\end{abstract}

\maketitle

\section{Introduction} \label{Sec:intro}
\subsection{The absorption problem for random walks}
Let $S_k = \xi_1 + \dots + \xi_k$ be a random walk in $\R^d, d \ge 1$, with independent identically distributed (i.i.d.)\ increments $\xi_1, \xi_2, \dots$. We study the probability that the convex hull of the first $n$ steps of the walk does not contain the origin. In other words, the trajectory $S_1, \dots, S_n$ belongs to some open linear half-space (with $0$ at its boundary). This question is a natural generalization to higher dimensions of the problem to find the probability that a one-dimensional random walk does not change its sign by the time $n$. We will refer to $\P[0 \in \conv(S_1,S_2,\dots,S_n)]$ as to the {\it absorption probability} and to the question of its computation as to the {\it absorption problem}.

The probability that a one-dimensional random walk stays positive (or negative) was fully understood by the mid-1950's. There were no results on the absorption problem for random walks in higher dimensions until the very recent papers of Eldan~\cite{eldan_extremal}, Tikhomirov and Youssef~\cite{tikhomirov_youssef2}, Vysotsky and Zaporozhets~\cite{vysotsky_zaporozhets}. This is despite of the fact that convex hulls of multidimensional random walks, Brownian motions, and other L\'evy processes are very popular objects of studies; we refer to~\cite{vysotsky_zaporozhets} for references, including the general surveys on this broad subject.

Our first result is as follows.

\begin{theorem} \label{theo:intro}
Let $S_1, \dots, S_n$ be a random walk in $\R^d$ whose i.i.d.\  increments $\xi_1,\dots,\xi_n$ have a centrally symmetric distribution (i.e., $\xi_1 \eqdistr - \xi_1$), and suppose additionally that $\P[\xi_1\in H] = 0$ for every affine hyperplane $H\subset \R^d$. Then
\begin{equation} \label{eq:intro}
\P[0 \notin \conv(S_1,S_2,\dots,S_n)]
=
\frac{2}{2^n n!}
\sum_{k=1}^{\lceil d/2\rceil} B(n,d - 2k + 1),
\end{equation}
where $B(n,k)$ are the coefficients of the polynomial
$$
(t+1)(t+3)\dots (t+2n-1) = \sum_{k=0}^n B(n,k) t^k.
$$
\end{theorem}
It is remarkable that in any dimension, the above absorption probabilities are distribution-free, i.e., independent of the distribution of increments of the walk.
 This fact was conjectured by Vysotsky and Zaporozhets~\cite{vysotsky_zaporozhets}, who found an explicit formula for the absorption probabilities in the planar case $d=2$. Specifying to $d=1$ and noticing that $B(n, 0) =(2n-1)!!$, we recover the famous distribution-free result of Sparre Andersen for the probability that a random walk with continuous symmetric distribution of increments stays positive:
\begin{equation} \label{eq:Sparre}
\P[S_1>0,\dots, S_n>0] = \frac{(2n-1)!!}{2^n n!} = \frac {1}{2^{2n}}\binom{2n}{n}.
\end{equation}

We will refer to the condition that $\P[\xi_1\in H] = 0$ for every affine hyperplane $H\subset \R^d$ as to the {\it general position assumption} since it implies that any $d$ random vectors among $S_1, \ldots, S_n$ are linearly independent a.s.; see Proposition~\ref{prop:iid_sufficient_condition} below. We will impose this or similar assumptions in most of our results. However, even without it we can show that the absorption probability for {\it any} $n$-step symmetric random walk\footnote{By saying that a random walk is symmetric we always mean the central symmetry of the distribution of its increments.} (for which $\P[\xi_1\in H]$ may be positive) is lower-bounded by the right-hand side in \eqref{eq:intro}; see Proposition~\ref{prop:non_general_absorbtion_probab} below.

\subsection{The equivalent geometric problems} \label{subsec:cor}
Our proof rests on a newly established direct connection between the probabilistic problem and an equivalent geometric problem concerning Weyl chambers. We solve the geometric problem using the theory of hyperplane arrangements and find the absorption probabilities explicitly. This method is entirely different from that of~\cite{vysotsky_zaporozhets}, and it even requires a noticeable effort to check that the formulas for the absorption probabilities match for $d=2$ (we will not include the details but wish to thank the anonymous referee for showing us this derivation).

Let us state the equivalent geometric problem. A {\it Weyl chamber} of type $B_n$ is any of the $2^n n!$ convex cones in $\R^n$ of the form
$$\{(x_1,\dots,x_n)\in\R^n\colon 0< \varepsilon_1 x_{\sigma(1)}< \varepsilon_2 x_{\sigma(2)}<\dots <\varepsilon_n x_{\sigma(n)}\},$$
where $\sigma(1), \dots, \sigma(n)$ is a permutation of $1, \dots, n$ and $\varepsilon_1, \dots, \varepsilon_n \in \{-1, 1\}.$ Equivalently, the Weyl chambers are the regions in $\R^n$ constituting the complement of the arrangement of $n^2$ hyperplanes $x_i =0$ and $x_i \pm x_j = 0, i \neq j$, which form the hyperfaces of the chambers.

It turns out that under the assumptions of Theorem~\ref{theo:intro}, we have
\begin{equation} \label{eq:intro_reduction}
\P[0 \in \conv(S_1,S_2,\dots,S_n)] = \frac{N_{n,d}}{2^n n!},
\end{equation}
where $N_{n,d}$ is the {\it constant} number of Weyl chambers intersected by a generic non-random linear subspace of $\R^n$ of codimension $d$. The exact meaning of ``generic'' will be explained in Section~\ref{Sec:arrangements}, where we compute the value of $N_{n,d}$ using the formulas of Whitney and Zaslavsky from the theory of {\it hyperplane arrangements}; see Theorem~\ref{theo:regions_reflection_arrangement}.

There is other connection between the absorption problem and a problem in spherical convex geometry: we show that for symmetric random walks whose distribution of increments puts no mass on affine hyperplanes,
$$
\P[0 \notin \conv(S_1,S_2,\dots,S_n)] =
\frac{2}{2^n n!} \sum_{k=1}^{\lceil d/2\rceil} \upsilon_{d-2k+1}(W_n),
$$
where $W_n$ is a Weyl chamber of type $B_n$ and $\upsilon_k$ are the {\it conic intrinsic volumes}. Hence finding the latter quantities solves the probabilistic problem. Conic intrinsic volumes are the analogues in conic geometry for intrinsic volumes of convex sets in Euclidean geometry. Intrinsic volumes include such fundamental geometric characteristics of a convex set as its volume, surface area, and mean width.

The described connections between the geometric problems via the absorption problem  gave us the insight to obtain a new simple proof of the general formula by Klivans and Swartz~\cite{klivans_swartz} that relates the coefficients of the characteristic polynomial of a linear hyperplane arrangement to the conic intrinsic volumes of the chambers constituting its complement; see Theorem~\ref{2201}. We used this result to find the conic intrinsic volumes $\upsilon_k(W_n)$ of Weyl chambers of type $B_n$; see Theorem~\ref{theo:conic_vol_W}. The particular value $\upsilon_1(W_n)$, which corresponds to the planar absorption probability, was found in~\cite{vysotsky_zaporozhets} (the conic hull of the orthoscheme path-simplex considered there is exactly the standard Weyl chamber). Let us mention the very recent papers by~\citet{AL15}  and~\citet{Schneider16} that consider further extensions of the Klivans--Swartz formula~\cite{klivans_swartz}. Both works appeared after the first version of the present paper; the approach of \cite[Theorem 1.2]{Schneider16} extends our proof of Theorem~\ref{2201}.

The explicit formula given in Theorem~\ref{theo:intro} allows one to find easily the asymptotics of the absorption probability in a fixed dimension $d$ as the number of steps $n$ tends to infinity; see Theorem~\ref{theo:asympt_fixed_d}. It also allows us to do  asymptotic analysis as the dimension increases and the number of steps $n=n(d)$ grows accordingly to make the absorption occur with a non-trivial probability. In this case we prove that the absorption probabilities follow a central limit theorem; see Theorem~\ref{theo:CLT}. In particular, our result shows that the phase transition from the non-absorption to absorption occurs as the number of steps reaches $n \approx e^{2d}$. In Theorem~\ref{theo:LD} we give the sharp asymptotics for the absorption and non-absorption probabilities in the respective large deviations regions $n\approx e^{2d/c}$ and $n\approx e^{2dc}$ for any $c>1$. This refines the much less precise bounds for the large deviations regions by Eldan~\cite{eldan_extremal} and by Tikhomirov and Youssef~\cite{tikhomirov_youssef2}; see the discussion in Section~\ref{subsec:LD}. We also obtain a version of this result for simple random walks, which of course do not satisfy the general position assumption. In this case a one-sided bound for the absorption probability in large deviation regions is given in Theorem~\ref{theo:simple_RW_estimate}.


\subsection{Extensions to other types of increments}
The Coxeter group $B_n$ is the symmetry group of the regular cube $[-1,1]^n$. The $2^n n!$ elements of this group act on $\R^n$ by permuting the coordinates in arbitrary way and multiplying any number of coordinates by $-1$. This is a finite reflection group generated by reflections along the hyperfaces of any Weyl chamber of type $B_n$. Every Weyl chamber $W_n$ of type $B_n$ is a fundamental region for the action of $B_n$: this means that the sets $g W_n, g \in B_n,$ are disjoint and their closures constitute the entire $\R^n$. We refer to~\cite{benson_grove} for an introduction to finite reflection groups.

Our method actually applies to convex hulls of not only random walks with i.i.d.\ increments but also of any sequence of partial sums $S_k= \xi_1 + \dots + \xi_k, k=1, \dots, n$, if their increments $\xi_1, \dots, \xi_n$ are possibly dependent random vectors in $\R^d$ whose joint distribution is invariant under the action of $B_n$, that is
\begin{equation}\label{eq:intro_symm_exch}
(\xi_1,\dots,\xi_n) \eqdistr (\eps_1 \xi_{\sigma(1)},\dots, \eps_n \xi_{\sigma(n)})
\end{equation}
for any permutation $\sigma(1), \dots, \sigma(n)$ of $1, \dots, n$ and any $\varepsilon_1, \dots, \varepsilon_n \in \{-1, 1\}.$ In this case we say that the tuple $(\xi_1,\dots,\xi_n)$ is {\it symmetrically exchangeable}. There are many important examples of such tuples with non-i.i.d.\ entries: say, for $d=1$, the tuple of coordinates of any rotationally invariant non-Gaussian random vector in $\R^n$ is so. The exact statement of Theorem~\ref{theo:intro} extended to symmetrically exchangeable increments is given in Theorem~\ref{theo:convex_hull_B_n} of the next section.

It turns out that our approach can be generalized to solve the absorption problems for partial sums of increments whose joint distribution is invariant under the action of other finite groups. These are the reflection groups of types $A_{n-1}$ and $D_n$, and direct products of finite reflection groups. Let us explain.

The Coxeter group $A_{n-1}$ is the symmetry group of the regular simplex (defined as the convex hull of the standard basis vectors in $\R^n$). The $n!$ elements of this group act on $\R^n$ by permuting the coordinates. The tuple $(\xi_1, \dots, \xi_n)$ of random vectors in $\R^d$ is called {\it exchangeable} if its distribution is invariant under the action of $A_{n-1}$, that is
\begin{equation}\label{eq:intro_exch}
(\xi_1,\dots,\xi_n) \eqdistr (\xi_{\sigma(1)},\dots, \xi_{\sigma(n)})
\end{equation}
for any permutation $\sigma(1), \dots, \sigma(n)$ of $1, \dots, n$. We will use the standard notation $\Sym(n)$ for the symmetric group on $\{1, \dots, n\}$ of such permutations.

The action of $A_{n-1}$ leaves the hyperplane $x_1+\dots+x_{n} = 0$ invariant. Restricting the action of $A_{n-1}$ to this hyperplane allows us to apply the method described above in Section~\ref{subsec:cor} to the convex hulls $\conv(S_1, \dots, S_{n-1})$ of partial sums with exchangeable increments that satisfy the condition $\xi_1 + \dots + \xi_n = 0$ a.s. This covers the absorption problem for {\it random walk bridges} under the corresponding general position assumption. The respective analogue of \eqref{eq:intro_reduction} rests on counting the number of Weyl chambers of type $A_{n-1}$ intersected by a generic subspace of codimension $d$ of the hyperplane $x_1+\dots+x_{n} = 0$. The exact statement is given in Theorem~\ref{theo:convex_hull_A}.

The Coxeter group $D_n$ is the subgroup of $B_n$ of index two whose action on $\R^n$ changes signs of an even number of coordinates. The corresponding Theorem~\ref{theo:convex_hull_D_n}, concerning the joint distribution of increments invariant under the action of $D_n$, applies to the convex hulls of symmetric random walks that are allowed to choose the sign of the last jump; see the discussion in Section~\ref{sec:types_ABD}.

Thus, all the results discussed above in Section~\ref{subsec:cor} are proved for the convex hulls of partial sums of increments whose joint distribution is invariant under the action of one of the finite reflection groups $A_{n-1}$, $B_n$, or $D_n$. The corresponding equivalent geometric statements concern Weyl chambers of the respective  types.

Finally, we consider the case when the increments are invariant under the action of a direct product  of reflection groups. We restrict ourselves to direct products of groups of type $B$. The corresponding Theorem~\ref{theo:convex_hull_product} solves the absorption problem for the {\it joint} convex hull of several symmetric random walks with possibly different number of steps (under the corresponding general position assumption).

In the particular case when all the random walks have the same distribution of increments and each walk has only one step, Theorem~\ref{theo:convex_hull_product} implies the well-known result of Wendel~\cite{wendel}: if $\xi_1, \dots, \xi_n$ are i.i.d.\ random vectors in $\R^d$ with an absolutely continuous centrally symmetric distribution, then
\begin{equation} \label{eq:Wendel}
\P[0 \notin \conv(\xi_1, \dots, \xi_n)] = \frac{1}{2^{n-1}}\sum_{k=0}^{d-1}\binom{n-1}{k}.
\end{equation}

Thus, our approach brings together the classical distribution-free results~\eqref{eq:Sparre} and~\eqref{eq:Wendel} by Sparre Andersen and Wendel, respectively, on symmetrically distributed random variables.

\medskip

Let us explain the structure of the paper. Section~\ref{sec:types_ABD} contains the explicit formulas for the  absorption probabilities under the different types of increments. These results are proved in Section~\ref{Sec:proofs}. In Sections~\ref{Sec:arrangements} and~\ref{Sec:instinsic} we provide some basic facts from the theory of hyperplane arrangements and conic convex geometry, and prove our new results on the geometric problems equivalent to the absorption problem. The asymptotic analysis of absorption probabilities is given in Section~\ref{Sec:asympt}. We conclude the paper with the list of open questions.

\section{Main results: convex hulls of random walks and bridges}\label{sec:types_ABD}

In this section we present the explicit formulas for the  absorption probabilities for the  partial sums
$$S_k= \xi_1 + \dots + \xi_k, \quad 1 \le k \le n$$
with the joint distribution of their increments $\xi_1, \dots, \xi_n \in \R^d$  invariant under the action of the finite reflection groups $A_{n-1}$, $B_n$, $D_n$, and the analogous result for the invariance under direct products.

\subsection{Type \texorpdfstring{$A_{n-1}$}{A\_\{n-1\}}: Random walk bridges}
The Coxeter group $A_{n-1}$ is the symmetric group $\Sym(n)$, which acts on $\R^{n}$ by permuting the coordinates. The number of elements in this group is $n!$. The action of this group leaves the hyperplane
$$
L= \{(x_1,\dots,x_{n})\in \R^{n}\colon x_1+\dots+x_{n} = 0\}
$$
invariant. This explains why the subscript $n-1$ rather than $n$ appears in the standard notation $A_{n-1}$.
Note that the group $A_{n-1}$ is the symmetry group of the regular simplex with $n$ vertices, i.e.\ the convex hull of the standard basis in $\R^{n}$.
\begin{theorem}\label{theo:convex_hull_A}
Let $(\xi_1,\dots,\xi_{n})$ be an exchangeable tuple (see \eqref{eq:intro_exch}) of random vectors in $\R^d$ with partial sums $S_1,\dots,S_{n}$. Assume that $S_n=0$ a.s.,\ $n\geq d+1$, and any $d$ random vectors among $S_1, \ldots, S_{n-1}$ are linearly independent a.s. Then
\begin{equation}\label{eq:probab_conv_hull_A}
\P[0 \in \conv(S_1,\dots,S_{n-1})] =
\frac 2 {n!} \left(\stirling{n}{d+2} + \stirling{n}{d+4}+\dots \right),
\end{equation}
where $\stirling{n}{k}$ are the Stirling numbers of the first kind defined by the formula
\begin{equation}\label{eq:def_stirling}
t(t+1) \dots (t+n-1) = \sum_{k=1}^n \stirling{n}{k} t^k
\end{equation}
with the convention that $\stirling{n}{k}=0$ for $k\notin \{1,\dots, n\}$.
\end{theorem}
\begin{remark}
The sum in~\eqref{eq:probab_conv_hull_A} is (as well as many other sums of a similar type appearing below) contains only finitely many non-zero terms.  
Combining~\eqref{eq:probab_conv_hull_A} with the identity
\begin{equation}\label{eq:stirling_even_odd}
\stirling{n}{1} + \stirling{n}{3} +\dots =\stirling{n}{2} + \stirling{n}{4} +\dots=\frac {n!}2
\end{equation}
(which can be obtained by taking $t=\pm 1$ in~\eqref{eq:def_stirling}), yields the following formula for the probability of non-absorption:
\begin{equation}\label{eq:probab_conv_hull_A_non_absorb}
\P[0 \notin \conv(S_1,\dots,S_{n-1})] =
\frac 2 {n!} \left(\stirling{n}{d} + \stirling{n}{d-2}+\dots \right).
\end{equation}
\end{remark}
In the one-dimensional case $d=1$ we obtain from~\eqref{eq:probab_conv_hull_A_non_absorb} that the probability that $S_1,\dots,S_{n-1}$ do not change their sign is
\begin{equation}\label{eq:sparre_andersen_bridge}
\P[S_1,\dots,S_{n-1} > 0 \text{ or } S_1,\dots,S_{n-1} < 0]
=
\frac 2n,
\end{equation}
since $\stirling{n}{1} = (n-1)!$. In fact, Sparre Andersen~\cite[Corollary~2]{Sparre1953}
showed that the probability of staying positive and the probability of staying negative are $1/n$ each. Theorem~\ref{theo:convex_hull_A} can be viewed as a multidimensional generalization of this classical formula.



The number $\frac 1 {n!}\stirling{n}{k}$ turns out to be the $k$th conic intrinsic volume of the Weyl chamber of type $A_{n-1}$, as shown below in Section~\ref{subsec:conic_intrinsic_weyl_chambers}. It can be also interpreted as the probability of having $k$ records in $n$ i.i.d.\ observations from a continuous distribution~\cite[Lecture 13]{nevzorov_book}, or as
$$
\frac1 {n!} \stirling{n}{k} = \P[\delta_1+\dots +\delta_n = k],
$$
where $\delta_1,\dots,\delta_n$ are independent random variables (the record indicators) with $\delta_i \sim \text{Bernoulli} (\frac 1 {i})$.

It is remarkable that Theorem~\ref{theo:convex_hull_A} (as well as the other similar theorems stated below) is distribution-free, that is the probability in~\eqref{eq:probab_conv_hull_A} does not depend on the distribution of $\xi_1,\dots,\xi_n$. No moment conditions on the random vectors are imposed.

Let us stress that without the general position condition imposed in Theorem~\ref{theo:convex_hull_A}, the absorption probabilities become distribution-dependent. For example, for the bridge of a simple random walk $S_k$ on $\Z$ (which makes jumps $\pm 1$ with probability $1/2$), it is known that  for any even $n$, it holds
$$
\P[S_1,\dots,S_{n-1} > 0 \text{ or } S_1,\dots,S_{n-1} < 0 \,|\, S_n=0] = \frac{1}{n-1},
$$
which is clearly different from~\eqref{eq:sparre_andersen_bridge}.

\subsection{Type \texorpdfstring{$B_n$}{B\_n}: Symmetric random walks}
The Coxeter group $B_n$ is the symmetry group of the regular cube $[-1,1]^n$ (or of its dual, the regular crosspolytope). The elements of this group act on $\R^n$ by permuting the coordinates in arbitrary way and multiplying any number of coordinates by $-1$. The number of elements of this group is $2^n n!$.
\begin{theorem}\label{theo:convex_hull_B_n}
Let $(\xi_1,\dots,\xi_n)$ be a symmetrically exchangeable tuple (see \eqref{eq:intro_symm_exch}) of random vectors in $\R^d$ with partial sums $S_1,\dots,S_n$. Assume that $n\geq d$ and any $d$ random vectors among $S_1, \ldots, S_n$ are linearly independent a.s. Then
\begin{equation}\label{eq:probab_conv_hull_B}
\P[0 \in \conv(S_1,S_2,\dots,S_n)] =
\frac{2}{2^n n!} (B(n,d + 1) + B(n, d+3) + \dots),
\end{equation}
where $B(n,k)$ are the coefficients of the polynomial
\begin{equation}\label{eq:def_b_nk}
(t+1)(t+3)\dots (t+2n-1) = \sum_{k=0}^n B(n,k) t^k
\end{equation}
and, by convention, $B(n,k) = 0$ for $k\notin \{0,\dots,n\}$.
\end{theorem}
\begin{remark}
By taking $t=\pm 1$ in~\eqref{eq:def_b_nk} we obtain the identity
\begin{equation}\label{eq:B_n_even_odd}
 B(n,1) + B(n,3) +\dots = B(n,0) + B(n,2) +\dots = 2^{n-1} n!.
\end{equation}
It follows that the probability of non-absorption  is given by
\begin{equation}\label{eq:probab_conv_hull_B_not_absorb}
\P[0 \notin \conv(S_1,S_2,\dots,S_n)] =
\frac{2}{2^n n!} (B(n,d - 1) + B(n, d - 3) + \dots).
\end{equation}
\end{remark}
The next proposition states a simple sufficient condition for Theorem~\ref{theo:convex_hull_B_n} to hold and yields Theorem~\ref{theo:intro}.
\begin{proposition}\label{prop:iid_sufficient_condition}
If $\xi_1,\xi_2,\dots$ are i.i.d.\  random vectors in $\R^d$ with partial sums $S_k = \xi_1+\dots +\xi_k$, $k\in\N$, then the following conditions are equivalent:
\begin{itemize}
\item [(i)]  for every $1\leq i_1 < \dots < i_d$, the random vectors $S_{i_1}, \dots, S_{i_d}$ are linearly independent with probability $1$;
\item [(ii)] for every affine hyperplane $H\subset \R^d$, we have $\P[\xi_1\in H] = 0$;
\item [(iii)] for every hyperplane $H_0\subset \R^d$ passing through the origin and every $i\in\N$, we have $\P[S_i \in H_0] =0$.
\end{itemize}
\end{proposition}
\begin{corollary}
If (ii) or (iii) is satisfied, then Theorem~\ref{theo:convex_hull_B_n} applies.
\end{corollary}
Importantly, the probability in~\eqref{eq:probab_conv_hull_B} does not depend on the distribution of the increments. This proves the conjecture of~\citet{vysotsky_zaporozhets} for general $d$. Specifying~\eqref{eq:probab_conv_hull_B_not_absorb} to $d=1$ and noting that $B(n,0)=(2n-1)!!$ we obtain the probability that a symmetric random walk with continuously distributed increments stays positive:
$$
\P[S_1>0,S_2>0,\dots, S_n>0] =
\frac{(2n-1)!!}{2^n n!} = \frac {1}{2^{2n}}\binom{2n}{n}.
$$
This recovers another classical result of Sparre Andersen~\cite{sparre_andersen0}. For a simple random walk $S_k$, the formula is different: by the reflection principle,
$$
\P[S_1>0,S_2>0,\dots, S_{n}>0] =
\frac 1 {2^{n}} \binom {n-1} {\left[\frac{n-1}2\right]}.
$$

The numbers $B(n,k)$ are called the $B$-analogs of the (signless) Stirling numbers of the first kind; see the entries  A028338 (or A039757 for the signed version) in~\cite{sloane}. They satisfy the recurrence relation
$$
B(n, k) = (2n-1) B(n-1,k) + B(n-1,k-1)
\quad
$$
and are explicitly given by
$
B(n,k) =  \sum_{i=k}^n  2^{n-i} \binom {i}{k} \stirling{n}{i}.
$
These numbers were studied in detail by~\citet{suter}. There is a probabilistic representation of $B(n,k)$: it follows directly from~\eqref{eq:def_b_nk} that
$$
\frac{B(n,k)}{2^n n!} = \P[\delta_1+\dots +\delta_n = k],
$$
where $\delta_1,\dots,\delta_n$ are independent random variables with $\delta_i \sim \text{Bernoulli} (\frac 1 {2i})$ for $1\leq i\leq n$. Geometrically, $B(n,k)/(2^n n!)$ is the $k$th conic intrinsic volume of the Weyl chamber of type $B_n$; see Section~\ref{subsec:conic_intrinsic_weyl_chambers}.

\subsection{Type \texorpdfstring{$D_n$}{D\_n}}
The Coxeter group $D_n$ acts on $\R^n$ by permuting the coordinates in an arbitrary way and by multiplying any even  number of coordinates by $-1$. It is a subgroup of $B_n$ of index $2$ and the number of its elements is $2^{n-1} n!$. $D_n$ is the symmetry group of the demihypercube constructed from alternation of the regular cube $[-1,1]^n$.

\begin{theorem}\label{theo:convex_hull_D_n}
Let $\xi_1,\dots,\xi_n$ be random vectors in $\R^d$ such that for every permutation $\sigma\in \Sym(n)$ and every $\eps_1,\dots,\eps_n\in \{-1,+1\}$ with $\eps_1\dots\eps_n = +1$,
\begin{equation}\label{eq:xi_symm_D}
(\xi_1,\dots,\xi_n) \eqdistr (\eps_1 \xi_{\sigma(1)},\dots, \eps_n \xi_{\sigma(n)}).
\end{equation}
Let $S_1,\dots,S_n$ denote the partial sums of $\xi_1,\dots,\xi_n$, and put $S_n^*= S_{n-1} - \xi_n$. Assume that $ n \ge \max\{2, d\}$ and any $d$ random vectors from either collection $S_1, \ldots, S_n$ or $S_1, \ldots, S_{n-1}, S_n^*$ are linearly independent a.s. Then
\begin{equation}
\P[0 \in \conv(S_1,\dots,S_{n-1}, S_n, S_n^*)] =
\frac{2}{2^{n-1} n!} (D(n, d+1) + D(n, d+3) + \dots),
\end{equation}
where $D(n,k)$ are the coefficients of the polynomial
\begin{equation}\label{eq:def_d_nk}
(t+1)(t+3)\dots (t+2n-3)(t+n-1) = \sum_{k=0}^n D(n,k) t^k
\end{equation}
and, by convention, $D(n,k) = 0$ for $k\notin \{0,\dots,n\}$.
\end{theorem}
\begin{remark}
The probability of non-absorption  is given by
\begin{equation}\label{eq:probab_conv_hull_D_not_absorb}
\P[0 \notin \conv(S_1,\dots, S_{n-1}, S_n, S_n^*)] =
\frac{2}{2^{n-1} n!} (D(n,d - 1) + D(n, d - 3) + \dots).
\end{equation}
\end{remark}

For example, Theorem~\ref{theo:convex_hull_D_n} can be applied when $\xi_1,\dots,\xi_n$ are i.i.d.\ random vectors as in Proposition~\ref{prop:iid_sufficient_condition}.
It is easy to show (see \eqref{eq:U_conv} below) that for any $n \ge 2$,
$$
\conv (S_1,\dots, S_{n-1}, S_n, S_n^*) = \conv(S_1,\dots, S_{n-1}, S_{n}) \cup \conv (S_1,\dots, S_{n-1}, S_n^*),
$$
hence the probabilistic problem corresponding to the symmetry group $D_n$ concerns the convex hull of a symmetric random walk allowed to ``choose'' the sign of its last jump in order to absorb the origin.

The numbers
$$
D(n,k)= (n-1)B(n-1,k) + B(n-1,k-1)
$$
are called the $D$-analogs of the (signless) Stirling numbers of the first kind; see the entry A039762 in~\cite{sloane} for the signed version. It will be shown in  Section~\ref{subsec:conic_intrinsic_weyl_chambers} that $D(n,k)/(2^{n-1}n!)$ is the $k$th conic intrinsic volume of the Weyl chamber of type $D_n$. Moreover, we have
$$
\frac{D(n,k)}{2^{n-1} n!} = \P[\delta_1+\dots +\delta_n = k],
$$
where $\delta_1,\dots,\delta_n$ are independent random variables with $\delta_i \sim \text{Bernoulli} (\frac 1 {2i})$ for $1\leq i \leq n-1$ and $\delta_n \sim \text{Bernoulli} (\frac 1 n)$.

\subsection{Direct products of reflection groups}
So far we considered probabilistic problems related to irreducible reflection groups. It is known that a general reflection group can be represented as direct sum of irreducible ones. In this section we study the absorption problem for the \textit{joint} convex hull of \textit{several} random walks and/or random walk bridges. The corresponding symmetry groups are the direct products of finite reflection groups.

To be specific, we restrict ourselves to direct products of the form $B_{n_1}\times \dots \times B_{n_r}$ containing only groups of the same type, namely $B$. Here $r \in\N$ corresponds to the number of random walks and $n_i$, where $1 \le i \le r$, stands for the number of steps in the $i$th walk. It is straightforward to extend our results to products of the form $A_{n_1}\times \dots \times  A_{n_r}$, which corresponds to joint convex hulls of several random walk bridges, and even to mixed direct products containing groups of all $3$ types $A, B, D$. We omit such extension because it requires more complicated notation.


\begin{theorem}\label{theo:convex_hull_product}
Let $\xi_1^{(1)}, \dots, \xi_{n_1}^{(1)}, \dots, \xi_1^{(r)}, \dots, \xi_{n_r}^{(r)}$ be random vectors in $\R^d$ such that for
every permutations $\sigma^{(1)}\in \Sym(n_1), \dots, \sigma^{(r)} \in \Sym(n_r)$ and every signs $\eps_1^{(1)}, \dots, \eps_{n_1}^{(1)},$ $\dots, \eps_1^{(r)}, \dots, \eps_{n_r}^{(r)} \in \{-1, +1\}$, we have
\begin{multline}\label{eq:invar_product}
(\xi_1^{(1)}, \dots, \xi_{n_1}^{(1)}, \dots, \xi_1^{(r)}, \dots, \xi_{n_r}^{(r)})
\\\eqdistr
(\eps_1^{(1)} \xi_{\sigma_1(1)}^{(1)}, \dots, \eps_{n_1}^{(1)} \xi_{\sigma_1(n_1)}^{(1)}, \dots, \eps_{1}^{(r)}\xi_{\sigma_r(1)}^{(r)}, \dots, \eps_{n_r}^{(r)}\xi_{\sigma_r(n_r)}^{(r)}).
\end{multline}
Let $S_k^{(i)} = \xi_1^{(i)} + \dots + \xi_{k}^{(i)}, \,  1\leq i \leq r,\,  1\leq k \leq n_i$, denote the partial sums.  Assuming that $n_1+\dots+n_r\geq d$ and any $d$ random vectors from the collection $S_1^{(1)}, \dots, S_{n_1}^{(1)},\dots, S_1^{(r)}, \dots, S_{n_r}^{(r)}$ are linearly independent a.s., we have
\begin{multline}
\P[0\in \conv (S_1^{(1)}, \dots, S_{n_1}^{(1)},\dots, S_1^{(r)}, \dots, S_{n_r}^{(r)})]
=
\frac{2(P(d+1) + P(d+3)+\dots)}{2^{n_1} n_1!\dots 2^{n_r} n_r!} ,
\end{multline}
where the $P(k)$'s (which also depend on $r,n_1,\dots,n_r$) are the coefficients of the polynomial
\begin{equation}\label{eq:def_p_k}
\prod_{i=1}^r ((t+1)(t+3)\dots (t+2n_i-1)) = \sum_{k=0}^{n_1+\dots+n_r} P(k) t^k 
\end{equation}
and $P(k)=0$ for $k\notin\{0,\dots, n_1+\dots+n_r\}$.
\end{theorem}
Since the proof of Theorem~\ref{theo:convex_hull_product} is based on the same ideas as the proofs in the irreducible cases, but requires complicated notation, it will be presented elsewhere.


\begin{example}[Type $B_1^r$: The Wendel formula] \label{ex:Wendel}
Let us consider the particular case $n_1=\dots=n_r=1$ that all random walks make just one step. This corresponds to the direct product of $r$ groups $\Z/2\Z$, where each factor acts on $\R$ by multiplication by $\pm 1$.

The random vectors $\xi^{(1)}:=\xi_1^{(1)}, \dots, \xi^{(r)}:=\xi_1^{(r)}$ with values in $\R^d$ are required to satisfy
\begin{equation}\label{eq:invariance_signs}
(\xi^{(1)}, \dots, \xi^{(r)}) \eqdistr (\pm\xi_1^{(1)}, \dots, \pm\xi_1^{(r)})
\end{equation}
for all $2^r$ choices of the signs. Additionally, we assume that any $d$ of these $r$ random vectors are linearly independent a.s.
Then~\eqref{eq:def_p_k}, which defines $P(k)$'s, takes the form $(t+1)^r = \sum_{k=0}^r P(k) t^k$ so that $P(k) = \binom {r}{k}$. Theorem~\ref{theo:convex_hull_product} asserts that
$$
\P[0\notin \conv(\xi^{(1)}, \dots, \xi^{(r)}) ] =
\frac{1}{2^{r-1}} \left( \binom{r}{d-1} + \binom{r}{d-3} + \dots\right).
$$
Using the recursive property of the Pascal triangle, we obtain
$$
\P[0\notin \conv(\xi^{(1)}, \dots, \xi^{(r)}) ] = \frac{1}{2^{r-1}} \sum_{k=0}^{d-1} \binom {r-1} k.
$$
This formula is due to~\citet{wendel}, whose proof is essentially based on Schl\"afli's formula~\eqref{eq:Schlafli} presented below; see~\cite[Section~8.2.1]{schneider_weil_book}.

The same result can be obtained if one considers the symmetry group $A_1^r$ since its action on $\R^{2r}$ is isomorphic to the action of $B_1^r$ on $\R^r$.
\end{example}

\begin{remark}
Although some of our arguments can be extended to other group representations, such extensions do not seem to have a natural probabilistic interpretation. Here is the most meaningful example: by considering the direct product of $r$ dihedral groups, it is possible to find the probability of absorption of the origin by the convex hull of $r$ sides chosen uniformly at random in $r$ regular polygons centred at the origin. We prefer to omit such results here.
\end{remark}


\subsection{Removing the general position assumptions}
As explained above, the general position assumption is essential in our results. Without this assumption,  it is still possible to obtain a one-sided bound  for the absorption probabilities.

Let $(\xi_1,\dots,\xi_n)$ be a tuple of random vectors in $\R^d$ that satisfies all the assumptions of any of Theorems~\ref{theo:convex_hull_A}, \ref{theo:convex_hull_B_n}, or~\ref{theo:convex_hull_D_n}. Denote by $H_{n,d}$ the  convex hull considered in the respective theorem. Let $(\xi_1',\dots,\xi_n')$ be any tuple of random vectors in $\R^d$ that satisfies all the assumptions of the corresponding theorem {\it except} the general position one. Put $S_k'=\xi_1' + \dots + \xi_k, 1 \le k \le n,$ and $(S_n^*)'=\xi_1'+\dots + \xi_{n-1}' - \xi_n'$, and denote by $H_{n,d}'$ the  convex hull of the respective type.

Note that both $H_{n,d}'$ and $H_{n,d}$ are closed, and denote by ${\rm{Int}}(H_{n,d}')$ the interior of $H_{n,d}'$.

\begin{proposition}\label{prop:non_general_absorbtion_probab}
For any of the cases $A_{n-1}, B_n, D_n$, we have
\begin{align}
\P[0 \in {\rm{Int}}(H_{n,d}')]  \leq \P[0 \in H_{n,d}] \leq  \P[0 \in H_{n,d}']. \label{eq:H_n_d_ineq}
\end{align}

\end{proposition}

In particular, this result covers simple random walks on $\Z^d$, where $\xi_1',\dots,\xi_n'$ are i.i.d.\ and $\P[\xi_1'=e_i] = \P[\xi_1'=-e_i] = \frac{1}{2d}$ for $ i=1, \dots, d$, with $e_1, \dots, e_d$ being the standard basis in $\R^d$.
\begin{proof}
Since the absorption probability is distribution-free under the respective general position assumption, we can assume without loss of generality that  $\xi_i=\xi_i' + \eps \delta_i$, where $\varepsilon \neq 0$ and $\delta_1,\dots,\delta_n$ are  random vectors in $\R^d$ independent of $\xi_1',\dots,\xi_n'$, with the following distribution. In the  $B_n$ and $D_n$ cases, $\delta_1,\dots,\delta_n$ are  i.i.d.\ standard normal vectors in $\R^d$, whereas in the $A_{n-1}$ case, they are i.i.d.\ standard normal vectors in $\R^d$ conditioned on $\delta_1+\dots+\delta_n=0$. The tuple $(\xi_1,\dots,\xi_n)$ defined in this way satisfies the assumptions of the respective  Theorem~\ref{theo:convex_hull_A}, \ref{theo:convex_hull_B_n}, or~ \ref{theo:convex_hull_D_n}.

Note that the convex hull $H_{n,d}$ is obtained from $H_{n,d}'$ by a small random distortion.
We have
\begin{align}
\P[0\in H_{n,d}] &\leq \P[0\in H_{n,d}'] + \P[0 \in H_{n,d}, 0 \notin H_{n,d}'], \label{eq:ineq_H_n_d_1}\\
\P[0\in H_{n,d}] &\geq \P[0\in {\rm {Int}} (H_{n,d}')] - \P[0 \in {\rm{Int}} (H_{n,d}'), 0 \notin H_{n,d}], \label{eq:ineq_H_n_d_2}
\end{align}
where $\P[0\in H_{n,d}]$ does not depend on $\varepsilon$, whereas
$$
\lim_{\eps\to 0} \P[0 \in H_{n,d}, 0 \notin H_{n,d}'] = \lim_{\eps\to 0} \P[0 \in {\rm{Int}} (H_{n,d}'), 0 \notin H_{n,d}] = 0
$$
since $H_{n,d}'$ is a closed set and ${\rm{Int}} (H_{n,d}')$ is an open set. Letting $\eps\to 0$ in~\eqref{eq:ineq_H_n_d_1} and~\eqref{eq:ineq_H_n_d_2} proves~\eqref{eq:H_n_d_ineq}. Note that the difference in probabilities in~\eqref{eq:H_n_d_ineq} can occur because if $0$ is on the boundary of $H_{n,d}'$, then even a small distortion possibly gets $H_{n,d}$ aside of $0$.
\end{proof}

In Section~\ref{subsec:proof_without_density} we will present another proof of Proposition~\ref{prop:non_general_absorbtion_probab} which follows our geometric interpretation in terms of intersections of Weyl chambers. It is easy to extend Proposition~\ref{prop:non_general_absorbtion_probab} to direct products of reflection groups.


\section{Hyperplane arrangements} \label{Sec:arrangements}
\subsection{The main formula for the number of regions}
A \emph{linear hyperplane arrangement} (or simply ``\textit{arrangement}'') $\cA$ is a finite set of distinct hyperplanes in $\R^n$ that pass through the origin. The literature on hyperplane arrangements~\cite{OT92}, \cite{stanley_book} considers the more general concept of \textit{affine} hyperplane arrangements (the hyperplanes are not required to pass through the origin) but in the present work we study only the linear case.

The \emph{rank} of an arrangement $\cA$, denoted by $\rank(\cA)$, is the dimension of the linear subspace spanned by the normals to the hyperplanes in $\cA$. Equivalently, the rank is the codimension of the intersection of all hyperplanes in the arrangement:
$$
\rank(\cA)=n-\dim\left(\bigcap_{H\in\cA}H\right).
$$
The \emph{characteristic polynomial} $\chi_{\cA}(t)$ of the arrangement $\cA$ is defined by
\begin{equation}\label{1459}
\chi_{\cA}(t)=\sum_{\cB\subset\cA}(-1)^{\#\cB}t^{n-\rank(\cB)},
\end{equation}
where the sum is over all subsets $\cB$ of $\cA$,  $\#$ denotes the number of elements, and $\rank(\varnothing) = 0$ under convention that the intersection over the empty set of hyperplanes is $\R^n$. The original definition of the characteristic polynomial is much more complicated and uses the notions of the intersection poset of $\cA$ and the M\"obius function on it; see~\cite[Section~1.3]{stanley_book}. For our purposes we need only the above equivalent definition. The equivalence was proved by Whitney; see, e.g., \cite[Lemma~2.3.8]{OT92} or~\cite[Theorem~2.4]{stanley_book}.

Denote by $\cR(\cA)$  the set of open connected components (``regions'' or ``chambers'') of the complement $\R^n\setminus\cup_{H\in\cA} H$ of the hyperplanes. 
The following fundamental result due to Zaslavsky~\cite{tZ75} (see also~\cite[Theorem~2.5]{stanley_book}) expresses the number of regions of the arrangement $\cA$ in terms of its characteristic polynomial:
\begin{equation}\label{1112}
\# \cR(\cA)=(-1)^n\chi_{\cA}(-1).
\end{equation}


Let $\cA$ be an arrangement in $\R^n$ and let $L_{n-d}$ be a linear subspace in $\R^n$ of codimension $d\leq n-1$. We say that $L_{n-d}$ is in \emph{general position} with respect to $\cA$ if for every non-empty subset $\cB\subset\cA$
\begin{equation}\label{1222}
\dim\left(\bigcap_{H\in\cB}(H\cap L_{n-d})\right)=
\begin{cases}
n-d-\rank(\cB), &\text{if } \rank(\cB)\leq n-d,\\
0, &\text{if } \rank(\cB)\geq n-d.
\end{cases}
\end{equation}

Our aim is to find a formula for the number of regions in $\cR(\cA)$ intersected by $L_{n-d}$. Consider the induced arrangement $\cA|L_{n-d}$, that is the arrangement in $L_{n-d} \cong \R^{n-d}$ defined by\footnote{In this definition we assume that the linear subspace $L_{n-d}$ is in general position w.r.t.\ $\cA$ and that $n-d\neq 1$. This ensures that every $H\cap L_{n-d}$ has codimension $1$ in $L_{n-d}$ (by~\eqref{1222}) and that all these hyperplanes are distinct. Indeed, if $H_1\cap L_{n-d}= H_2\cap L_{n-d}$, then both subspaces have dimension $n-d-1$ by~\eqref{1222}, but, on the other hand, $H_1\cap H_2$ has dimension $d-2$ and hence, $H_1\cap H_2 \cap L_{n-d}$ has dimension $n-d-2\geq 0$ by~\eqref{1222}, which is a contradiction. In the case that $L_{n-d}$ is a line in general position w.r.t.\ $\cA$, we define $\cA|L_{n-d} = \{\{0\}\}$.}
$$
\cA|L_{n-d}=\{H\cap L_{n-d} \colon H\in\cA\}.
$$
It is not hard to show, using the fact that $R \cap L_{n-d}$ is connected in $L_{n-d}$ for every $R \in \cR(\cA)$, that
the regions of the induced arrangement are obtained by intersecting the regions of $\cA$ with $L_{n-d}$. Then, clearly, we have
$$
\#\{R\in \cR(\cA)\colon R\cap L_{n-d}\ne\varnothing\}=\#\cR(\cA|L_{n-d}).
$$

\begin{lemma}\label{2148}
Let $\cA$ be a linear hyperplane arrangement in $\R^n$ and let $L_{n-d}$ be a linear subspace in $\R^n$ of codimension $d\leq n-1$ that is in general position w.r.t.\ $\cA$. Let
\begin{equation}\label{eq:chi_def}
\chi_{\cA}(t)=\sum_{k=0}^n (-1)^{n-k} a_kt^k
\end{equation}
be the characteristic polynomial of $\cA$. Then the characteristic polynomial of $\cA$ restricted to $L_{n-d}$ is given by
\begin{equation}\label{eq:chi_restriction}
\chi_{\cA|L_{n-d}}(t)=\sum_{k=0}^{d} (-1)^{n-k} a_k+\sum_{k=d+1}^n (-1)^{n-k} a_kt^{k-d}.
\end{equation}
\end{lemma}
\begin{remark} \label{rem:coeff}
It is easy to show  that $a_n=1$, $a_{n-1}=\# \cA$;
see~\cite[p.~400]{stanley_book}.
Moreover, the sequence  $a_0,\dots, a_n$ is strictly positive~\cite[Corollary~3.5]{stanley_book} and
unimodal~\cite[Lecture 2, Exercise~9 on p.~419]{stanley_book}.
Let us also prove the identity
\begin{equation} \label{eq:even=odd}
a_0+a_2+\dots = a_1+a_3+\dots.
\end{equation}
By the second part of Zaslavsky's theorem~\cite[Theorem~2.5]{stanley_book}, for every affine hyperplane arrangement, the number of \emph{bounded} regions in $\cR(\cA)$ is (up to the sign) given by $\chi_{\cA}(1) = \sum_{k=0}^n(-1)^{n-k}a_k$. Since we are dealing only with linear hyperplane arrangements, there are no bounded regions, whence~\eqref{eq:even=odd}.
\end{remark}
\begin{proof}
If $L_{n-d}$ is a line, then $\cA|L_{n-d} = \{\{0\}\}$ and $\chi_{\cA|L_{n-d}} = t-1$, which is the same expression as in~\eqref{eq:chi_restriction} by~\eqref{eq:even=odd} and since $a_n=1$.

Suppose in the following that $n-d\geq 2$. It follows from~\eqref{1222} that for every subset $\cB\subset\cA$,
\begin{equation}\label{eq:rank_B_restricted}
\rank(\cB|L_{n-d})=
\begin{cases}
\rank(\cB), &\text{if } \rank(\cB)\leq n-d,\\
n-d, &\text{if } \rank(\cB)\geq n-d,
\end{cases}
\end{equation}
where the rank is in $L_{n-d}$. Also, as we explained in the footnote,
$
\#(\cB|L_{n-d}) = \#\cB
$
because $L_{n-d}$ is not a line.
Using~\eqref{1459} (in dimension $n-d$) and then~\eqref{eq:rank_B_restricted} we obtain
\begin{align*}
\chi_{\cA|L_{n-d}}(t)
&=
\sum_{\cB\subset\cA} (-1)^{\#\cB}t^{n-d-\rank(\cB|L_{n-d})}\\
&=
\sum_{k=0}^{d} \sum_{\substack{\cB \subset \cA\\\rank(\cB)=n-k}} (-1)^{\# \cB}
+\sum_{k=d+1}^n \sum_{\substack{\cB \subset \cA\\\rank(\cB)=n-k}} (-1)^{\# \cB} t^{n-d-\rank(\cB)}.
\end{align*}
After noting that by~\eqref{1459} and~\eqref{eq:chi_def},
$$
\sum_{\substack{\cB\subset A\\\rank(\cB)=n-k}} (-1)^{\# \cB} = (-1)^{n-k} a_k,
$$
we obtain the required formula.
\end{proof}
Now we are ready to state the main result of this section.
\begin{theorem}\label{1229}
Let $L_{n-d}$ be linear subspace in $\R^n$ of codimension $d$ that is in general position w.r.t.\ a linear hyperplane arrangement $\cA$. The number of regions in $\cR(\cA)$ intersected by $L_{n-d}$ is given by
$$
\#\{R\in \cR(\cA)\colon R\cap L_{n-d}\ne\varnothing\}
=
2(a_{d+1} + a_{d+3} +\dots),
$$
where the $a_k$'s are defined by~\eqref{eq:chi_def} and we set $a_k=0$ for $k\notin\{0,\dots,n\}$.
\end{theorem}
\begin{proof}
By~\eqref{1112} and Lemma~\ref{2148}, we have
$$
\#\{R\in \cR(\cA)\colon R\cap L_{n-d}\ne\varnothing\}
=
\begin{cases}
\sum_{k=0}^na_k - 2\sum_{k=0}^sa_{2k}, &\text{if } d=2s+1,\\
\sum_{k=0}^na_k - 2\sum_{k=1}^sa_{2k-1}, &\text{if } d=2s,
\end{cases}
$$
where we used that $\cA|L_{n-d}$ is an arrangement in dimension $n-d$. To complete the proof, recall~\eqref{eq:even=odd}.
\end{proof}




\subsection{Special case: the reflection arrangements}
The above results can be applied to the \emph{reflection arrangements} in $\R^n$ of the types $A_{n-1}$, $B_n$, $D_n$.
These arrangements consist  of the hyperplanes
\begin{align*}
\cA(A_{n-1})&\colon \quad  x_i = x_j, \quad 1\leq i < j \leq n,\\
\cA(B_n)&\colon \quad x_i = x_j, \quad x_i = -x_j, \quad x_k = 0, \quad 1\leq i < j \leq n, \quad 1\leq k\leq n,\\
\cA(D_n)&\colon \quad  x_i = x_j, \quad x_i = -x_j, \quad 1\leq i < j \leq n.
\end{align*}
\begin{theorem}\label{theo:regions_reflection_arrangement}
Let $L_{n-d}$ be a linear subspace in $\R^n$ of codimension $d$ that is in general position w.r.t.\ to one of the reflection arrangement $\cA(A_{n-1}), \cA(B_n), \cA(D_n)$. Then the number of regions in this arrangement intersected by $L_{n-d}$ is given, respectively, by
\begin{align*}
\cR(\cA(A_{n-1})|L_{n-d})
&=
2\left(\stirling{n}{d+1} + \stirling{n}{d+3}+\dots\right),
\\
\cR(\cA(B_{n})|L_{n-d}) &=
2(B(n, d+1) + B(n, d+3) +\dots),
\\
\cR(\cA(D_{n})|L_{n-d}) &=
2(D(n, d+1) + D(n, d+3) +\dots).
\end{align*}
\end{theorem}
\begin{proof}
The characteristic polynomials of the reflection arrangements are
(see Corollary 2.2 on p.~28 and Section 5.1 in~\cite{stanley_book})
\begin{align}
&\chi_{\cA(A_{n-1})}(t) = t (t-1) \dots (t-(n-1)) = \sum_{k=1}^{n} (-1)^{n-k} \stirling{n}{k} t^k,  \label{eq:chi_A} \\
\chi&_{\cA(B_{n})}(t) = (t-1)(t-3)\dots (t-(2n-1)) = \sum_{k=0}^{n} (-1)^{n-k} B(n,k)t^k, \label{eq:chi_B}\\
\notag\chi_{\cA(D_{n})}(t) &= (t-1)(t-3)\dots (t-(2n-3))(t-(n-1)) = \sum_{k=0}^{n} (-1)^{n-k} D(n,k) t^k,
\end{align}
where we have used~\eqref{eq:def_stirling}, \eqref{eq:def_b_nk}, \eqref{eq:def_d_nk}. We stress that $\cA(A_{n-1})$ is an arrangement in $\R^n$, hence its characteristic polynomial has degree $n$.
\end{proof}

\subsection{Non-general position}
The following lemma compares the number of open and closed chambers intersected by an arbitrary linear subspace with the respective  number of chambers for a linear subspace in general position.
\begin{lemma}\label{lem:non-general_position}
Let $\cA$ be a linear arrangement in $\R^n$ and let $L_{n-d},L_{n-d}'$ be linear subspaces in $\R^n$ of codimension $d$.
If $L_{n-d}$ is in general position w.r.t.\ $\cA$, then
\begin{equation} \label{eq:R=bar_R}
\{R\in \cR(\cA)\colon \bar R\cap L_{n-d}\ne \{0 \}\} = \{R\in \cR(\cA)\colon R\cap L_{n-d}\ne\varnothing\}
\end{equation}
and
\begin{align}
&\#\{R\in \cR(\cA)\colon \bar R\cap L_{n-d}'\neq \{0\}\} \geq \#\{R\in \cR(\cA)\colon \bar R\cap L_{n-d}\neq  \{0\}\}, \label{eq:lem:non-general_position1}\\
&\#\{R\in \cR(\cA)\colon R\cap L_{n-d}'\ne\varnothing\} \leq \#\{R\in \cR(\cA)\colon R\cap L_{n-d}\ne\varnothing\}.
\label{eq:lem:non-general_position2}
\end{align}
\end{lemma}

The proof will be presented in Section~\ref{subsec:proof_without_density}.

\section{Connection with conic intrinsic volumes} \label{Sec:instinsic}
\subsection{Definition of conic intrinsic volumes}
We call a set $C\subset\R^n$ a \textit{convex cone} if for any $x,y\in C$ and $a,b>0$ it holds $ax+by\in C$. In the 1940's a spherical counterpart of the Steiner formula was developed in~\cite{cA48,gH43,lS50}. In its modern form (see~\cite[Section~6.5]{schneider_weil_book}, \cite[Section~IV]{lS76}, and~\cite{ALMT14, GNP14, MT14}), this formula expresses the size of angular expansions of a closed convex cone $C$ in $\R^n$:
\begin{equation}\label{1130}
\P[\dist^2(\theta, C)\leq\lambda] = \sum_{k=0}^n\beta_{k,n}(\lambda)\upsilon_k(C),
\end{equation}
where $\theta$ is a random variable uniformly distributed on the unit sphere $\S^{n-1}\subset \R^n$ and $\beta_{k,n}(\cdot)$ is the distribution function of a Beta distribution with parameters $(n-k)/2$ and $n/2$. Since the functions $\beta_{1,n},\dots,\beta_{n,n}$ are linearly independent, the formula defines the coefficients $\upsilon_k(C)$ uniquely. The quantities $\upsilon_0(C),\dots,\upsilon_n(C)$ are called the \emph{conic intrinsic volumes} of the cone $C$. The normalization is chosen so that these quantities do not depend on the dimension of the Euclidean space containing $C$, and thus conic intrinsic volumes do not change if we consider $C$ as naturally embedded into a space of higher dimension. Note that the $k$th conic intrinsic volume $\upsilon_k(C)$ equals the $(k-1)$th spherical intrinsic volume of $C\cap\S^{n-1}$ considered in \cite{GHS03,schneider_weil_book}.

Following the notation of~\cite{ALMT14}, for each $k\in \{0,\dots,n\}$, define the $k$th \emph{half-tail functional} by
\begin{equation}\label{2147}
h_k(C)=\upsilon_k(C)+\upsilon_{k+2}(C)+\dots,
\end{equation}
where we set $\upsilon_k(C)=0$ for $k\notin\{0,\dots, n\}$.


The conic intrinsic volumes satisfy a version of the Gauss--Bonnet theorem (see, e.g.,\ \cite[Theorem~6.5.5]{schneider_weil_book} or~\cite[p.~28]{ALMT14}): if $C$ is not a subspace, then
\begin{equation} \label{eq:Gauss-Bonnet}
h_0(C) = \upsilon_0(C) + \upsilon_2(C) +\ldots = \frac 12,
\quad
h_1(C) = \upsilon_1(C) + \upsilon_3(C) +\ldots = \frac 12.
\end{equation}
The conic analogue of the Crofton formula (see, e.g.,\ \cite[pp.~261--262]{schneider_weil_book} or~\cite[Equation~5.10]{ALMT14}) is the following relation: if $C$ is a closed convex cone that is not a subspace, then for every $d\in \{0,\dots, n-1\}$,
\begin{equation}\label{2133}
h_{d+1}(C)=\frac12\P[C\cap W_{n-d}\ne\{0\}],
\end{equation}
where $W_{n-d}$ is a random $(n-d)$-dimensional linear subspace in $\R^n$ chosen w.r.t.\ the uniform distribution on the Grassmannian.  

\subsection{Characteristic polynomial of linear arrangement and conic intrinsic volumes}
Let $\cA$ be a linear hyperplane arrangement in $\R^n$ with characteristic polynomial
$$
\chi_{\cA}(t)=\sum_{k=0}^n (-1)^{n-k} a_kt^k.
$$
The next theorem, conjectured by~\citet{drton_klivans} and proved by~\citet{klivans_swartz}, relates the coefficients of the characteristic polynomial to the conic intrinsic volumes of the regions of the arrangement. We will give a completely different (and very short) proof of this theorem. Our approach was already extended by Schneider~\cite[Theorem 1.2]{Schneider16}.  A generalization of the Klivans--Swartz formula~\cite{klivans_swartz}  was first considered by~\citet[Section 6]{AL15}, whose work appeared after the first version of our paper.
\begin{theorem}\label{2201}
For every linear hyperplane arrangement $\cA$ in $\R^n$,
$$
a_k = \sum_{R\in \cR(\cA)}\upsilon_k(R), \quad k=0,\dots,n.
$$
\end{theorem}
\begin{proof}
Let $W_{n-d+1}$ be a random $(n-d+1)$-dimensional linear subspace in $\R^n$ distributed according to the uniform measure on the Grassmannian, where $d\in \{1,\dots,n\}$.  With probability one, $W_{n-d+1}$ is in general position w.r.t.\ $\cA$. Thus, by Theorem~\ref{1229} we have
$$
\#\{R\in \cR(\cA)\colon R\cap W_{n-d+1}\ne\varnothing\}
=
2(a_{d} + a_{d+2} +\dots)\quad \text{a.s.}
$$
On the other hand, for any $R\in \cR(\cA)$ it follows from \eqref{2133} that
$$
h_{d}(R) = \frac12 \E [\ind_{ \{R\cap W_{n-d+1}\neq \varnothing \}}].
$$
Summing up over all $R\in \cR(\cA)$ and combining the formulas, we obtain that for all $d\in \{1,\dots,n\}$,
$$
\sum_{R\in \cR(\cA)}h_{d}(R)
=
\frac12 \E [\#\{R\in \cR(\cA)\colon R\cap W_{n-d+1}\ne\varnothing\}]
=
a_{d} + a_{d+2} +\dots.
$$
By \eqref{eq:even=odd}, \eqref{eq:Gauss-Bonnet}, and the equation above for $d=1$, it also holds
$$
\sum_{R\in \cR(\cA)}h_{0}(R)
=
a_0+a_2+ \dots.
$$
Using the formula $\upsilon_k(R) = h_{k}(R) - h_{k+2}(R)$ (which follows from~\eqref{2147}) we obtain that for all $k\in \{0,\dots,n\}$,
$$
\sum_{R\in \cR(\cA)} \upsilon_k(R) = \sum_{R\in \cR(\cA)} h_{k}(R) - \sum_{R\in \cR(\cA)} h_{k+2}(R) = a_k.
$$
This completes the proof.
\end{proof}

\subsection{Conic intrinsic volumes of the Weyl chambers}\label{subsec:conic_intrinsic_weyl_chambers}
The {\it Weyl chambers} of type $A_{n-1}$, $B_n$, $D_n$ are the following convex cones in $\R^n$:
\begin{align*}
\cC(A_{n-1}) &:=\{(x_1,\dots,x_{n})\in \R^n \colon x_1<x_2<\dots < x_{n}\},\\
\cC(B_n) &:= \{(x_1,\dots,x_n)\in\R^n\colon 0< x_1<x_2<\dots < x_n\},\\
\cC(D_n) &:= \{(x_1,\dots,x_n)\in\R^n\colon 0 < |x_1| < x_2 < \dots < x_n\}.
\end{align*}
Each Weyl chamber $C=\cC(G)$ is a {\it fundamental domain} for the corresponding reflection group $G=A_{n-1}, B_n$ or $D_n$. This means that all cones of the form $gC, g \in G$, which are also called Weyl chambers without any risk of confusion, are disjoint and that $\cup_{g \in G} g \bar{C} = \R^n$ holds true.
\begin{theorem} \label{theo:conic_vol_W}
The conic intrinsic volumes of the Weyl chambers of types $A_{n-1}$, $B_n$, $D_n$ are given by
$$
\upsilon_k(\cC(A_{n-1})) = \frac 1{n!} \stirling{n}{k}, \quad
\upsilon_k(\cC(B_n)) = \frac {B(n,k)}{2^n n!}, \quad
\upsilon_k(\cC(D_n)) = \frac {D(n,k)}{2^{n-1} n!},
$$
for $k=0,1,\dots,n$, where $\stirling{n}{k}$, $B(n,k)$, $D(n,k)$ are as in~\eqref{eq:def_stirling}, \eqref{eq:def_b_nk}, \eqref{eq:def_d_nk}, respectively.
\end{theorem}
\begin{proof}
To be specific, consider the $\cA_{n-1}$ case. The coefficients of the characteristic polynomial of the corresponding hyperplane arrangement are $a_k = \stirling{n}{k}$; see the proof of Theorem~\ref{theo:regions_reflection_arrangement}. The regions in $\cR(\cA_{n-1})$ are the $n!$ isometric Weyl chambers of the type $A_{n-1}$. By Theorem~\ref{2201} we obtain
$$
\stirling{n}{k} = n! \upsilon_k(\cC(A_{n-1})),
$$
which proves the required formula. The $B_n$ and $D_n$ cases are analogous.
\end{proof}

\subsection{Random arrangements of hyperplanes in general position}
Let $\cA$ be a linear arrangement in $\R^n$ consisting of $m\geq n$ hyperplanes in general position; in our terminology, this means that $\R^n$ is in general position w.r.t.\ $\cA$; see~\eqref{1222}. By~\eqref{1459} we have that
\begin{equation}\label{1706}
\chi_{\cA}(t)=\sum_{k=0}^{n-1}(-1)^{k}{\binom m k}t^{n-k}+\sum_{k=n}^m (-1)^{k}{\binom m k}.
\end{equation}
Applying~\eqref{1112} and using that the alternating binomial coefficients add up to zero, we get
$$
\cR(\cA)=\sum_{k=0}^{n-1}{\binom mk}+\sum_{k=n}^m(-1)^{k+n}{\binom m k}= \sum_{k=0}^{n-1}{\binom m k} - \sum_{k=0}^{n-1}(-1)^{k+n}{\binom m k},
$$
hence by the recursive property of the Pascal triangle,
\begin{equation} \label{eq:Schlafli}
\cR(\cA)= 2\sum_{k=0}^{n-1}{\binom {m-1}{k}}=:C(m,n).
\end{equation}
This well-known formula, proved by Schl\"afli~\cite[pp.209--212]{Schl50} for a general dimension, goes back to Steiner~\cite{jS26} for $n=3$; see also~\cite[Lemma~8.2.1]{GHS03} for a simple inductive proof and references. We already saw this formula in Example~\ref{ex:Wendel}.

Let $X_1,\dots,X_m$ be i.i.d.\ random vectors on the unit sphere $\S^{n-1}$ such that their common distribution is centrally symmetric and assigns no mass to any $(n-2)$-dimensional great subsphere. The hyperplanes $X_1^\perp,\dots,X_m^\perp$, which are in general position a.s.,\ divide $\R^n$ into $C(m,n)$ random cones. We choose one of these cones uniformly at random to obtain the  \emph{random Schl\"afli cone} $C_m$ in $\R^n$ introduced by Hug and Schneider in~\cite{HS15}.

The next result of~\cite{HS15} (given there with a slightly different notation) calculates the expected intrinsic volumes of a random Schl\"afli cone. This theorem easily follows from Theorem~\ref{2201}.
\begin{theorem}
For any random Schl\"afli cone $C_m$ in $\R^n$, it holds
$$
\E\,\upsilon_k(C_m)=\frac1{C(m,n)}{\binom m {n-k}},\quad k=0,\dots,n.
$$
\end{theorem}
\begin{proof}
Let $\cA$ be the arrangement consisting of the hyperplanes $X_1^\perp,\dots,X_m^\perp$. We have
$$
\E\,\upsilon_k(C_m)=\frac1{C(m,n)}\E\,\sum_{R\in\cR(\cA)}\nu_k(R).
$$
On the other hand, it follows from Theorem~\ref{2201} and~\eqref{1706} that
$$
\sum_{R\in\cR(\cA)}\upsilon_k(R) = {\binom m {n-k}}\quad \mbox{a.s.},
$$
which completes the proof.
\end{proof}
\begin{remark}
As readily seen from the proof, this result holds true for any \emph{deterministic}  vectors $X_1,\dots,X_m$ that are in general position. The essential randomness here is in the uniform measure on the $C(m,n)$ elements of $\cR(\cA)$.
\end{remark}

\section{Asymptotic results} \label{Sec:asympt}
In this section we use the exact expressions of Section~\ref{sec:types_ABD} to study the asymptotic behavior of the probability that the convex hull of a symmetric random walk or of a random walk bridge absorbs the origin. The cases $A_{n-1}$, $B_n$ and $D_n$ are very similar. We work in the setting of Theorems~\ref{theo:convex_hull_A}, \ref{theo:convex_hull_B_n},  \ref{theo:convex_hull_D_n}, which assume the exchangeability of increments, the assumption of general position, and  in the cases $B_n$ and $D_n$, the symmetry of the distribution of increments. Recall that $H_{n,d}\subset \R^d$ denotes the convex hull considered in any of these theorems.

\subsection{Asymptotics in constant dimension}
The following theorem gives the asymptotics of the non-absorption probability in the case that the dimension $d$ is fixed and the number of steps $n$ tends to infinity. In the case $d=2$ this result was obtained by~\citet{vysotsky_zaporozhets}.  We write $x_n \simn y_n$  if $\lim_{n\to\infty} x_n/y_n =1$.

\begin{theorem}\label{theo:asympt_fixed_d}
For any fixed dimension $d\geq 2$, under the assumptions of any of Theorems~\ref{theo:convex_hull_A}, \ref{theo:convex_hull_B_n}, \ref{theo:convex_hull_D_n}, it holds
$$
\P[0\notin H_{n,d}] \simn
\begin{cases}
\frac{2(\log n)^{d-1}}{(d-1)! n}, &\text{in the $A_{n-1}$ case},\\
\frac{(\log n)^{d-1}}{2^{d-2} (d-1)!\sqrt {\pi n}}, &\text{in the $B_n$ and $D_n$ cases}.\\
\end{cases}
$$
\end{theorem}
\begin{proof}
In the $A_{n-1}$ case, we can use the well-known asymptotics of the Stirling numbers~\cite{wilf}: for a fixed $k\in\N$,
$$
\frac{1}{(n-1)!}\stirling{n}{k} \simn \frac{(\log n)^{k-1}}{(k-1)!}.
$$
Substituting this formula into the statement of Theorem~\ref{theo:convex_hull_A} (see also~\eqref{eq:probab_conv_hull_A_non_absorb}) and noting that the term $\stirling{n}{d}$ dominates all other terms, we obtain
$$
\P[0\notin H_{n,d}]
=
\frac 2{n!} \left(\stirling{n}d  + \stirling{n}{d-2}+\dots\right)  \simn \frac 2{n!}\stirling{n}d \simn \frac{2(\log n)^{d-1}}{(d-1)! n}.
$$
In the $B_n$ case, the definition of $B(n,k)$ given in~\eqref{eq:def_b_nk} yields
\begin{align*}
B(n,k)
&=
(2n-1)!! \sum_{1\leq i_1 <\dots < i_k\leq n} \frac{1}{(2i_1-1)\dots (2i_k-1)}\\
&\simn
\frac{(2n-1)!!}{k!} \left(1+\frac 13 +\dots +\frac{1}{2n-1}\right)^k\\
&\simn
\frac {(2n-1)!!} {2^k k!} (\log n)^{k},
\end{align*}
where, in order to pass from the first line to the second one, one shows (by an omitted standard argument) that the contribution of all terms in the sum where at least two indices $i_m$ and $i_l$ are equal is $o((\log n)^k)$.
Substituting the above asymptotics for $B(n,k)$ into the statement of Theorem~\ref{theo:convex_hull_B_n} (see also~\eqref{eq:probab_conv_hull_B_not_absorb}), and noting that the term $B(n, d-1)$ dominates all other terms, we obtain
$$
\P[0\notin H_{n,d}] = \frac 2 {2^n n!}(B(n,d-1) + B(n,d-3)+\dots)  \simn \frac {2B(n,d-1)} {2^nn!}.
$$
Using the asymptotics of $B(n,d-1)$ and the Stirling formula, we obtain
$$
\P[0\notin H_{n,d}]
\simn
2 \frac {(2n-1)!!} {2^nn!} \frac {(\log n)^{d-1}}{{2^{d-1} (d-1)!}}
\simn
\frac{(\log n)^{d-1}}{2^{d-2} (d-1)!\sqrt {\pi n}}.
$$
The computation for the $D_n$ case is similar to the one for the $B_n$ case and yields the same result.
\end{proof}

\subsection{High-dimensional asymptotics: central limit theorem} \label{Subsec:CLT}
Consider the convex hull $H_{n,d}$ of a symmetric random walk (or any random walk bridge) of length $n$ in a \textit{high} dimension $d$. It is clear that if $n$ is sufficiently small, then the absorption probability should be close to $0$, whereas for sufficiently large $n$ the absorption probability should be close to $1$. Hence, at some value of $n$ (which is a function of $d$) there should be a phase transition from non-absorption to absorption. This transition was studied by~\citet{eldan_extremal} and in the subsequent paper by~\citet{tikhomirov_youssef2}.

In this section we provide a  precise description of the location of this phase transition. It will be convenient for us to make $d=d(n)$ a function of $n$ rather than considering $n$ as a function of $d$. The next theorem shows, in particular, that the absorption probabilities $\P[0\in H_{n,d(n)}]$ exhibit a phase transition at $d(n) \approx \frac 12 \log n$ for symmetric random walks and $d(n)\approx \log n$ for random walk bridges.
\begin{theorem}\label{theo:CLT}
Let the dimension $d=d(n)$ be such that for some $a\in\R$,
$$
d(n) = u \log n + a \sqrt{u \log n} + o(\sqrt{\log n}),
$$
as $n\to\infty$, where $u=1$ in the $A_{n-1}$ case and $u = \frac12$ in the $B_n$ and $D_n$ cases.
Then under the assumptions of any of Theorems~\ref{theo:convex_hull_A}, \ref{theo:convex_hull_B_n}, \ref{theo:convex_hull_D_n},
\begin{equation}\label{eq:theo:CLT}
\lim_{n\to\infty} \P[0 \notin H_{n,d(n)}] = \frac 1 {\sqrt{2\pi}}\int_{-\infty}^a \eee^{-t^2/2} \dd t.
\end{equation}
\end{theorem}
\begin{proof}
Consider the case $A_{n-1}$ first. By Theorem~\ref{theo:convex_hull_A} (see~\eqref{eq:probab_conv_hull_A_non_absorb}), we have
\begin{equation}\label{eq:probab_conv_hull_A_non_absorb_copy}
\P[0 \notin H_{n, d(n)}]
=
\frac 2 {n!} \left(\stirling{n}{d} + \stirling{n}{d-2}+\dots \right).
\end{equation}
A classical result of Goncharov (see, e.g.,\ \cite[Sec.\ IX.5]{Flajolet_book} or~\cite[p.~63]{nevzorov_book}) states that the Stirling numbers of the first type satisfy a central limit theorem (CLT) of the form
\begin{equation} \label{eq:Goncharov_CLT}
\lim_{n\to\infty} \frac 1 {n!} \sum_{k=1}^{d(n)} \stirling{n}{k} =   \frac 1 {\sqrt{2\pi}}\int_{-\infty}^a \eee^{-t^2/2} \dd t.
\end{equation}

On the other hand, the Stirling numbers of the first kind are unimodal (in $k$) being the (signed) coefficients of the characteristic polynomial of a hyperplane arrangement; see Remark~\ref{rem:coeff} and \eqref{eq:chi_A}. Combining this fact with \eqref{eq:Goncharov_CLT} and the unimodality of the standard normal density, we see that the mode $m_n$ of the sequence $\stirling{n}{k}$ satisfies
\begin{equation} \label{eq:mode}
m_n = u \log n + o(\sqrt{\log n}).
\end{equation}
Hence if $a<0$, then $\stirling{n}{k}$ are monotone increasing in $k$ for $k < d(n)$, and thus
\begin{equation} \label{eq:absrp=Stirling}
\frac 1 {n!} \sum_{k=1}^{d(n)} \stirling{n}{k} \le \P[0 \notin H_{n,d(n)}] \le \frac 1 {n!} \sum_{k=1}^{d(n)+1} \stirling{n}{k}.
\end{equation}

By the Goncharov CLT \eqref{eq:Goncharov_CLT}, this implies the required \eqref{eq:theo:CLT} for $a<0$. The proof for $a>0$ follows analogously by considering the complement probabilities $\P[0 \in H_{n,d(n)}]$ and using the monotonicity of $\stirling{n}{k}$ for $k> d(n)$. Finally, the case $a=0$ is easily treated using the continuity of the standard normal density. This completes the proof of Theorem~\ref{theo:CLT} in the case $A_{n-1}$.

We now turn to the case $B_n$. We will use the powerful theory of mod-Poisson convergence developed in~\cite{feray_meliot_nikeghbali, kowalski_nikeghbali}. Once established, the mod-Poisson convergence yields many limit theorems besides the CLT.

Let $X_n$ be an integer-valued  random variable with the distribution
\begin{equation}\label{eq:X_n_distr_def}
\P[X_n = k] = \frac{1}{2^n n!} B(n, k), \quad k=0,\dots, n.
\end{equation}
Note that the probabilities indeed sum up to one by~\eqref{eq:def_b_nk} with $t=1$.

We claim that $X_n$ satisfies a central limit theorem of the form
\begin{equation}\label{eq:CLT_X_n}
\frac{X_n - \frac12 \log n}{\sqrt{\frac 12 \log n}} \todistr \mathrm{N} (0,1),
\end{equation}
where $\mathrm{N} (0,1)$ is the standard normal law. This is analogous to \eqref{eq:Goncharov_CLT}.

Denote by $(x)_n = x (x + 1) \dots (x + n - 1)$ the rising factorial. By the definition of $B(n,k)$ given in~\eqref{eq:def_b_nk}, the moment generating function of $X_n$ is
$$
\E [\eee^{z X_n}] = \frac 1 {2^{2n} n!} \cdot \frac{(\eee^z)_{2n}}{(\frac 12 \eee^z)_n} .
$$
Recall that $(x)_n \sim n^x \Gamma(n)/\Gamma(x)$ as $n\to\infty$. This holds locally uniformly in $x\in\C$ and follows from the Weierstrass infinite product formula for $1/\Gamma(x)$. Using this asymptotics and the Stirling formula, we obtain the following.
\begin{lemma}
For the sequence $X_n$ defined in~\eqref{eq:X_n_distr_def}, locally uniformly on $\C$ we have
\begin{equation}\label{eq:mod_poisson}
\lim_{n\to\infty}  \frac{\E [\eee^{z X_n}]}{\eee^{(\frac 12 \log n)(\eee^z - 1)}} =
\frac {2^{\eee^z}\Gamma(\frac 12 \eee^z)} {2 \sqrt{\pi} \Gamma(\eee^z)}.
\end{equation}
\end{lemma}

The denominator $\eee^{(\frac12 \log n)(\eee^z - 1)}$ on the left-hand side is the moment generating function of a Poisson distribution with parameter $\frac 12 \log n$. Hence \eqref{eq:mod_poisson} states that $X_n$ converges in the mod-Poisson sense. This implies the CLT~\eqref{eq:CLT_X_n} by the general theory of mod-Poisson convergence; see~\cite[Proposition~2.4, Part~(2)]{kowalski_nikeghbali}.

The rest of the proof is completely analogous to the case $A_{n-1}$. The probability mass function of $X_n$ is unimodal by Remark~\ref{rem:coeff} and~\eqref{eq:chi_B}, and its mode satisfies~\eqref{eq:mode} (with $u=\frac 12$) by the established CLT~\eqref{eq:CLT_X_n}. By Theorem~\ref{theo:convex_hull_B_n} (see \eqref{eq:probab_conv_hull_B_not_absorb}) combined with the definition of $X_n$ given in \eqref{eq:X_n_distr_def} and the unimodality of $B(n,k)$, we see that if $a<0$, then
\begin{equation} \label{eq:absrp=B_Stirling}
\P [X_n \leq  d(n)-1 ] \le \P[0\notin H_{n,d(n)}] \le \P [X_n \leq  d(n)].
\end{equation}
This is analogous to \eqref{eq:absrp=Stirling} and proves the required \eqref{eq:theo:CLT} in the $B_n$ case for $a<0$. The case $a \ge 0$ is covered as above.

The $D_n$ case is completely analogous to the $B_n$ case and yields the same asymptotics.
\end{proof}

\begin{corollary}
It holds
$$
\lim_{n\to\infty} \P[0\notin H_{n,d(n)}] =
\begin{cases}
0,  & \text{ if }  \limsup_{n\to\infty} \frac{d(n)}{u\log n} < 1, \\
1,  & \text{ if }  \liminf_{n\to\infty} \frac{d(n)}{u\log n} > 1, \\
1/2,& \text{ if }  d(n) = u \log n + o(\sqrt{\log n}). 
\end{cases}
$$
\end{corollary}
\begin{proof}
The third case follows by taking $a=0$ in Theorem~\ref{theo:CLT}. The other two cases also follow from Theorem~\ref{theo:CLT} since the probability $\P[0\notin H_{n,d}]$ is increasing in $d$ if we restrict $d$ to be even or odd; see Theorems~\ref{theo:convex_hull_A}, \ref{theo:convex_hull_B_n}, \ref{theo:convex_hull_D_n}.
\end{proof}

\subsection{High-dimensional asymptotics: large deviations} \label{subsec:LD}


In this section we give the asymptotics for the absorption (non-absorption) probabilities in  regions of large deviations, where the random walk or bridge makes too few (respectively, too many) steps compared to a typical mode described in the previous section. The first results of this kind were obtained by \citet{eldan_extremal}, who proved that for some constants $0< c_1 <c_2$, non-absorption (respectively, absorption) occurs with a high probability provided that $n<\eee^{c_1 d/\log d}$ (respectively, $n> \eee^{c_2 d\log d}$). \citet{tikhomirov_youssef2} removed the $\log d$ factor and replaced the bounds above by $n<\eee^{c_1 d}$ and $n>\eee^{c_2 d}$, respectively.

The authors of \cite{eldan_extremal} and \cite{tikhomirov_youssef2} considered the following four models: a Brownian motion sampled  either at times $1,\dots,n$
or at the points of a homogeneous Poisson point process restricted to $[0,1]$;
a simple random walk;
and a Rayleigh random walk (whose i.i.d. increments are uniformly distributed on the unit sphere $\S^{d-1}$).
Our  result presented below is sharp and holds true for any increments that satisfy assumptions of any of Theorems~\ref{theo:convex_hull_A}, \ref{theo:convex_hull_B_n}, \ref{theo:convex_hull_D_n}. In particular, it is valid for a Rayleigh random walk and for a Brownian motion sampled at times $1,\dots,n$. Further, our result can be easily adapted to a Brownian motion sampled at the jump times of a Poisson process since the increments in this model are exchangeable; in fact, conditioning on the number of jumps in $[0,1]$ makes the times between jumps exchangeable. However, without the general position assumption, we are able to cover only one of the two large deviation modes: our Theorem~\ref{theo:simple_RW_estimate} implies that the probability of non-absorption is polynomially small in the number of steps $n$ provided that $n > \eee^{(2+\eps) d}$ (for symmetric random walks) or $n > \eee^{(1+\eps) d}$ (for random walk bridges). Thus simple random walks on $\Z^d$ are not fully covered.

\begin{theorem}\label{theo:LD}
Suppose that $d(n)= u x_n \log n \in \N$ with $\lim_{n \to \infty} x_n=x$ for some constant $x >0$, where $u=1$ in the $A_{n-1}$ case and $u = \frac12$ in the $B_n$ and $D_n$ cases. Then under the assumptions of any of Theorems~\ref{theo:convex_hull_A}, \ref{theo:convex_hull_B_n}, \ref{theo:convex_hull_D_n},
\begin{align*}
&\P[0\notin H_{n,d(n)}]   \simn \frac{n^{- u(x_n\log x_n - x_n + 1)}}{\sqrt{2 \pi x u \log n}} \frac{L(x)}{1-x^{2u}},  \text{ if } x<1,\\
&\P[0 \in H_{n,d(n)}]   \simn \frac{n^{- u(x_n \log x_n - x_n + 1)}}{\sqrt{2 \pi x u \log n}} \frac{L(x)}{x^{2u}-1},  \text{ if } x>1,
\end{align*}
where $L(x)=\frac{2}{\Gamma(x)}$ in the $A_{n-1}$ case and $L(x)=\frac{2^x \sqrt{x} \Gamma(x/2)}{ \sqrt{\pi} \Gamma(x)}$ in the $B_n$ and $D_n$ cases.
\end{theorem}
\begin{remark}
Taking the first two terms of the Taylor series for $x\log x - x +1$ yields
\begin{align*}
&\P[0\notin H_{n,[u x \log n]}]   \simn \frac{n^{- u(x\log x - x + 1)}}{\sqrt{2 \pi x u \log n}} \frac{L(x) x^{\{u x \log n \}}}{1-x^{2u}},  \text{ if } x<1,\\
&\P[0 \in H_{n,[u x \log n]}]   \simn \frac{n^{- u(x \log x - x + 1)}}{\sqrt{2 \pi x u \log n}} \frac{L(x)  x^{\{u x \log n \}}}{x^{2u}-1},  \text{ if } x>1,
\end{align*}
where $\{y\}= y - [y]$ denotes the fractional part of a $y>0$.
\end{remark}
Note that the function $x\log x - x +1$, $x>0$, is the large deviations function of a standard Poisson distribution.
\begin{proof}
By Example 3.8 in~\cite{feray_meliot_nikeghbali}, for any fixed $k \in
\Z$ we have
$$\frac{1}{n!}\stirling{n}{x_n \log n + k } \simn \frac{n^{- (x_n\log x_n - x_n + 1)}}{\sqrt{2 \pi x \log n}} \frac{x^{-k}}{ \Gamma(x)}.$$
Similarly, by the general theory of mod-Poisson convergence~\cite[Theorem~3.4]{feray_meliot_nikeghbali}, which applies since the limit in the right-hand side of \eqref{eq:mod_poisson} is an entire analytic function non-vanishing for real $z$, we have
\begin{align*}
\frac{1}{2^n n!} B \Bigl(n, \frac12 x_n \log n + k \Bigr) &\simn \frac{n^{-\frac12 (x_n \log x_n - x_n + 1)}}{\sqrt{\pi x \log n}} \frac{2^x\Gamma(x/2)}{2 \sqrt{\pi} \Gamma(x)} x^{-k/2},\\
\frac{1}{2^{n-1} n!} D\Bigr(n, \frac12 x_n \log n + k \Bigr) &\simn \frac{n^{-\frac12 (x_n\log x_n - x_n + 1)}}{\sqrt{\pi x \log n}} \frac{2^x\Gamma(x/2)}{2 \sqrt{\pi} \Gamma(x)} x^{-k/2}.\\
\end{align*}
Then the claim follows by summation over even $k$ in the $A_{n-1}$ case or over odd $k$  in the $B_n$ and $D_n$ cases such that $k \ge 1$ if $x>1$ or $k \le 0$ if $x<1$. The summation is justified by the dominated convergence theorem and the second statement of Theorem~3.4 in~\cite{feray_meliot_nikeghbali}.
\end{proof}


Recall that for any tuple of random vectors $(\xi_1', \dots, \xi_n')$ in $\R^d$ that satisfies all the assumptions of any of Theorems~\ref{theo:convex_hull_A}, \ref{theo:convex_hull_B_n}, \ref{theo:convex_hull_D_n} {\it except} the general position one, $H_{n,d}'$ denotes the convex hull of the corresponding type.

\begin{theorem}\label{theo:simple_RW_estimate}
For every $\eps \in (0, \frac12)$ there exist $\delta=\delta(\eps)\in (0,1)$ and $C=C(\eps)>0$ such that for all $n > \eee^{d/(u-\eps)}$,
\begin{align*}
\P[0\in H_{n,d}'] \geq 1- C n^{-\delta},
\end{align*}
where $u=1$ in the $A_{n-1}$ case and $u = \frac12$ in the $B_n$ and $D_n$ cases.
\end{theorem}
\begin{proof}
In the case when the general position assumption holds, Theorem~\ref{theo:LD} implies that $\P[0\notin H_{n,d}] \leq Cn^{-\delta}$ for some $\delta \in (0,1)$ and $C>0$ since $d < (u-\eps) \log n$ and $x \log x -x +1$ is bounded away from $0$ if $x$ is bounded away from $1$. The claim follows by Proposition~\ref{prop:non_general_absorbtion_probab}.
\end{proof}

\medskip
It is natural to assume that Theorem~\ref{theo:simple_RW_estimate} is sharp in the following sense:
\begin{conjecture}
For every $\eps>0$ there exist $\delta=\delta(\eps)\in (0,1)$ and $C=C(\eps)>0$ such that for all $n < \eee^{d/(u+\eps)}$,
\begin{align*}
\P[0\notin H_{n,d}'] \geq 1- C n^{-\delta},
\end{align*}
where $u=1$ in the $A_{n-1}$ case and $u = \frac12$ in the $B_n$ and $D_n$ cases.
\end{conjecture}

\section{Proofs: Random convex hulls and Weyl chambers} \label{Sec:proofs}

In this section we prove our main probabilistic results Theorems~\ref{theo:convex_hull_A}, \ref{theo:convex_hull_B_n}, \ref{theo:convex_hull_D_n} on exact absorption probabilities under the respective general position assumption. Here we also prove Proposition~\ref{prop:non_general_absorbtion_probab} which estimates absorption probabilities for general random walks. The proofs are based on the same general approach but the particular details are rather different. For the reason of notation, we present together the proofs of Theorems~\ref{theo:convex_hull_A}, \ref{theo:convex_hull_B_n}, and \ref{theo:convex_hull_D_n}, and prove the other two results separately.

\subsection{Reflection groups \texorpdfstring{$A_{n-1}$, $B_n$, $D_n$}{A\_\{n-1\}, B\_n, D\_n}: Proofs of Theorems~\ref{theo:convex_hull_A}, \ref{theo:convex_hull_B_n}, and~\ref{theo:convex_hull_D_n}}

We identify the elements of the Coxeter groups $A_{n-1}$, $B_n$, and $D_n$ with orthogonal transformations $g: \R^n \to \R^n$. The  Weyl chambers $\cC(A_{n-1}), \cC(B_n)$, and $\cC(D_n)$ are respective fundamental domains for the actions of these groups. The action of $A_{n-1}$ leaves the hyperplane $L$ given by the equation $x_1+\dots+x_n=0$ invariant, hence $\cC(A_{n-1}) \cap L$ is a fundamental domain for the action of $A_{n-1}$ restricted on $L$.

Let $\xi_1,\dots,\xi_{n}$ be random vectors with values in $\R^d$ (written as columns), and let $A$ be the random $d\times n$-matrix with columns $\xi_1,\dots, \xi_{n}$. We regard $A$ as a random linear operator $A: \R^{n}\to\R^d$. The kernel $\Ker A$ of this operator is a random linear subspace of $\R^n$. Recall that $H_{n, d}$ is the common notation for the convex hulls considered in Theorems~\ref{theo:convex_hull_A}, \ref{theo:convex_hull_B_n}, \ref{theo:convex_hull_D_n}.

\begin{lemma} \label{lem:ABD}
Let $G$ be any of the Coxeter groups $A_{n-1}$, $B_n$, $D_n$, and let $C$ denote the respective domain $\cC(A_{n-1}) \cap L, \cC(B_n)$,  or $\cC(D_n)$. Suppose that the tuple $(\xi_1, \dots, \xi_n)$ satisfies all the assumptions of the respective Theorem~\ref{theo:convex_hull_A}, \ref{theo:convex_hull_B_n}, or \ref{theo:convex_hull_D_n} except the one on general position. Then for every $g \in G$,
$$
\P[0\in H_{n, d}] =  \P[\Ker A \cap  (g\bar C) \neq \{0\}].
$$
\end{lemma}

\begin{proof}
We are interested in the probability of the event
$$
E := \{\Ker A \cap (g\bar C) \neq \{0\}\}= \{ \Ker (Ag) \cap \bar C \neq \{0\}\}.
$$
Denote by $e_1,\dots,e_n$ the standard basis of $\R^n$, and recall that
$$
S_1= \xi_1, \;\; S_2=\xi_1+\xi_2, \;\; \dots,\;\; S_n=\xi_1+\dots+\xi_n.
$$

{\it Type $A_{n-1}$.} The elements of $A_{n-1}$ are the orthogonal transformations of the form $g_{\sigma}:\R^n\to\R^n$, where $\sigma\in \Sym(n)$ is a permutation on $n$ elements, and
$$
g_\sigma(e_k)= e_{\sigma(k)}, \quad k=1,\dots,n.
$$
It is easy to check that the columns of the matrix $Ag$ are $\xi_{\sigma(1)},\dots, \xi_{\sigma(n)}$. Hence,
$$
E = \{\exists x\in \bar C \bsl \{0\} \colon \xi_{\sigma(1)} x_1 + \dots + \xi_{\sigma(n)} x_{n} = 0\}.
$$

There is a bijective correspondence between $x = (x_1,\dots,x_{n})\in \bar C\bsl \{0\}$ and $y = (y_1,\dots,y_{n-1})\in \R_{\geq 0}^{n-1}\bsl \{0\}$ given by
$$
y_1= x_2-x_1, \;\; \dots, \;\; y_{n-1} = x_n - x_{n-1}
$$
or, equivalently,
$$
x_1=y_0, \;\; x_2=y_0+y_1,\;\; \dots, \;\;  x_{n}= y_0+\dots+y_{n-1},
$$
where $y_0\in\R$ is chosen to fulfill the condition $x_1+\dots+x_n = 0$. 
So the event $E$ occurs if and only if for some $y_0,\dots,y_{n-1}$ with the restrictions above,
$$
\xi_{\sigma(1)} y_0 + \xi_{\sigma(2)} (y_0+y_1) + \dots + \xi_{\sigma(n)} (y_0 + \dots + y_{n-1}) = 0.
$$
Rearranging the terms, we can write this as
$$
y_0 (\xi_{\sigma(1)} + \dots + \xi_{\sigma(n)})
+
y_1 (\xi_{\sigma(2)} + \dots + \xi_{\sigma(n)})
+ \dots +
y_{n-1}  \xi_{\sigma(n)} = 0.
$$
Using the assumption $\xi_1+\dots+\xi_{n} = 0$ a.s.,\ $y_0$ disappears and we can transform the above as
$$
y_1 \xi_{\sigma(1)}
+
y_2 (\xi_{\sigma(1)} + \xi_{\sigma(2)})
+ \dots +
y_{n-1} (\xi_{\sigma(1)} + \dots + \xi_{\sigma(n-1)})  = 0.
$$

The exchangeability assumption~\eqref{eq:intro_exch} on the distribution of $(\xi_1,\dots,\xi_{n})$ implies that
$$
(\xi_{\sigma(1)}, \xi_{\sigma(1)} + \xi_{\sigma(2)}, \dots, \xi_{\sigma(1)} + \dots + \xi_{\sigma(n-1)}) \eqdistr (S_1, S_2,\dots, S_{n-1}).
$$
Hence we obtain the required relation
\begin{align*}
\P[E]
&= \P[\exists (y_1,\dots, y_{n-1}) \in \R_{\geq 0}^{n-1}\bsl \{0\} \colon y_1 S_1 + y_2 S_2 + \dots + y_{n-1} S_{n-1} = 0]\\
&= \P[0\in \conv (S_1,S_2,\dots,S_{n-1})].
\end{align*}

{\it Type $B_n$.} The elements of $B_n$ are the orthogonal transformations of the form $g_{\sigma, \eps}:\R^n\to\R^n$, where $\sigma\in \Sym(n)$ is a permutation on $n$ elements, $\eps =(\eps_1,\dots,\eps_n)\in \{-1,+1\}^n$, and
$$
g_{\sigma,\eps} (e_k) = \eps_k e_{\sigma(k)}, \quad k=1,\dots,n.
$$
Note that the columns of the matrix $Ag$ are $\eps_1 \xi_{\sigma(1)},\dots, \eps_n \xi_{\sigma(n)}$, as one can see by computing $(Ag) e_1,\dots (Ag) e_n$. So we can write the event $E$ in the form
\begin{equation} \label{eq:E}
E = \{\exists x\in \bar C\bsl\{0\} \colon \eps_1 \xi_{\sigma(1)} x_1 + \dots + \eps_n \xi_{\sigma(n)} x_n = 0\}.
\end{equation}

There is a bijection between $x= (x_1,\dots,x_n)\in \bar C \bsl \{0\}$ and $y = (y_1,\dots,y_n)\in \R_{\geq 0}^n \bsl \{0\}$  given by
$$
x_1=y_1, \;\; x_2=y_1+y_2,\;\; \dots, \;\;  x_n= y_1+\dots+y_n.
$$
Hence we can write the condition for the event $E$ as
$$
\eps_1 \xi_{\sigma(1)} y_1 + \eps_2 \xi_{\sigma(2)} (y_1+y_2) + \dots + \eps_n \xi_{\sigma(n)} (y_1+\dots+ y_n) = 0,
$$
or equivalently,
$$
y_1 (\eps_1 \xi_{\sigma(1)} + \dots + \eps_n \xi_{\sigma(n)})
+
y_2 (\eps_2 \xi_{\sigma(2)} + \dots + \eps_n \xi_{\sigma(n)})
+ \dots +
y_n \eps_n \xi_{\sigma(n)} = 0.
$$

The symmetric exchangeability assumption~\eqref{eq:intro_symm_exch} on the distribution of $(\xi_1,\dots,\xi_n)$ implies that
$$
(\eps_1 \xi_{\sigma(1)} + \dots + \eps_n \xi_{\sigma(n)},\eps_2 \xi_{\sigma(2)} + \dots + \eps_n \xi_{\sigma(n)}, \dots, \eps_n \xi_{\sigma(n)}) \eqdistr (S_n, S_{n-1},\dots, S_1),
$$
hence we obtain the required
\begin{align*}
\P[E]
&= \P[\exists y\in \R_{\geq 0}^{n} \bsl \{0\} \colon y_1 S_n + y_2 S_{n-1} + \dots + y_n S_1 = 0]\\
&= \P[0\in \conv (S_1,S_2,\dots,S_n)],
\end{align*}

{\it Type $D_n$.} This case is very similar to the $B_n$ case as the elements of $D_n \subset B_n$ are the orthogonal transformations $g_{\sigma, \eps}$ such that $\eps_1 \dots \eps_n =1 $. There is a bijective correspondence between $x=(x_1,\dots,x_n)\in \bar C\bsl\{0\}$ and $y=(y_1,\dots,y_n)\in (\R \times \R_{\geq 0}^{n-1})\bsl\{0\}$ given by
$$
x_1=y_1, \;\; x_2=|y_1|+y_2,\;\; \dots, \;\;  x_n= |y_1|+y_2+\dots+y_n.
$$
So we can write the condition defining the event $E$ in \eqref{eq:E} in the form
$$
\eps_1 \xi_{\sigma(1)} y_1 + \eps_2 \xi_{\sigma(2)} (|y_1| + y_2) + \dots + \eps_n \xi_{\sigma(n)} (|y_1|+y_2+\dots+ y_n) = 0.
$$
Rearranging the terms, we obtain the equivalent representation
\begin{multline*}
|y_1| ((\sgn y_1) \eps_1 \xi_{\sigma(1)} + \eps_2 \xi_{\sigma(2)} + \dots + \eps_n \xi_{\sigma(n)})
\\+
y_2 (\eps_2 \xi_{\sigma(2)} + \dots + \eps_n \xi_{\sigma(n)})
+ \dots +
y_n \eps_n \xi_{\sigma(n)} = 0.
\end{multline*}

Recall that $S_n^* = \xi_1+\dots +\xi_{n-1} - \xi_n$. The invariance assumption~\eqref{eq:xi_symm_D} on the distribution of $(\xi_1,\dots,\xi_n)$ implies that
\begin{multline*}
(\eps_1 \xi_{\sigma(1)} + \dots + \eps_n \xi_{\sigma(n)}, -\eps_1 \xi_{\sigma(1)} + \dots + \eps_n \xi_{\sigma(n)},\\
\eps_2 \xi_{\sigma(2)} + \dots + \eps_n \xi_{\sigma(n)}, \dots, \eps_n \xi_{\sigma(n)}) \eqdistr (S_n, S_n^*, S_{n-1},\dots, S_1).
\end{multline*}
So we obtain that
\begin{multline*}
\P[E] = \P[\exists y \in \R_{\geq 0}^{n}\bsl\{0\} \colon y_1 S_n + y_2 S_{n-1} + \dots + y_n S_1 = 0
\\
\text{ or } y_1 S_n^* + y_2 S_{n-1} + \dots + y_n S_1 = 0],
\end{multline*}
hence
$$
\P[E]= \P[0\in \conv (S_1,\dots,S_{n-1}, S_n) \text{ or } 0\in \conv (S_1,\dots,S_{n-1},S_n^*)].
$$

To complete the proof of Lemma~\ref{lem:ABD}, we need to argue that
\begin{equation} \label{eq:U_conv}
\conv(S_1,\dots, S_{n-1}, S_{n}) \cup \conv (S_1,\dots, S_{n-1}, S_n^*) = \conv (S_1,\dots, S_n, S_n^*).
\end{equation}
The left-hand side is a subset of the right-hand side by definition of the convex hull. To see the converse inclusion, consider any convex combination
$$
x= \alpha_1 S_1 + \dots + \alpha_{n-2}S_{n-2} + \alpha_{n-1}S_{n-1} + \alpha_n S_n + \alpha_n^* S_n^*.
$$
If $\alpha_n\geq \alpha_n^*$, then by the identity $S_n^* = 2S_{n-1} - S_n$, we obtain
$$
x= \alpha_1 S_1 + \dots + \alpha_{n-2}S_{n-2} + (\alpha_{n-1} + 2\alpha_n^*) S_{n-1} + (\alpha_n-\alpha_n^*) S_n,
$$
which represents  $x$ as a convex combination of $S_1,\dots, S_{n-1}, S_{n}$. Indeed, the sum of the coefficients did not change, and hence equals one. This also ensures that the coefficients do not exceed one since all of them are non-negative.
Similarly, if $\alpha_n\leq \alpha_n^*$, we obtain  representation of $x$ as a convex combination of $S_1,\dots, S_{n-1}, S_{n}^*$.
\end{proof}

We are now ready to complete the proofs of Theorems~\ref{theo:convex_hull_A}, \ref{theo:convex_hull_B_n}, and \ref{theo:convex_hull_D_n}. Applying Lemma~\ref{lem:ABD} to all $g\in G$ and taking the arithmetic mean, we obtain
\begin{align}\label{eq:reduction}
\P[0\in H_{n, d}]
=
\frac 1 {\# G} \sum_{g\in G} \P[\Ker A \cap  (g\bar C) \neq \{0\}]
=
\frac {\E N} {\# G} ,
\end{align}
where the random variable
\begin{equation} \label{eq:N=}
N := \sum_{g\in G} \ind_{\{\Ker A \cap  (g\bar C) \neq \{0\}\}}
\end{equation}
counts the number of chambers of the form $g\bar C$, $g\in G$, intersected by $\Ker A$.

In the $A_{n-1}$ case, $N$ equals the number of closed Weyl chambers of type $A_{n-1}$ (non-trivially) intersected by $L\cap \Ker A$, as readily seen from the equation $\Ker A \cap  (g \bar C) = (L\cap \Ker A) \cap g(\overline{ \cC(A_{n-1})})$. In the $B_n$ and $D_n$ cases, $N$ equals the number of closed Weyl chambers of the respective type that have a non-trivial intersection with $\Ker A$.


It remains to use the following lemmas, whose proof is postponed for a moment.

\begin{lemma}[Type $A_{n-1}$]\label{lem:ker_A_gen_pos}
Under the assumptions of Theorem~\ref{theo:convex_hull_A}, $L\cap \Ker A$ a.s.\ has codimension $d+1$ in $\R^n$ and a.s.\ is in general position w.r.t.\ the arrangement $\cA(A_{n-1})$.
\end{lemma}
\begin{lemma}[Type $B_n$ and $D_n$]\label{lem:ker_BD_gen_pos}
Under the assumptions of Theorem~\ref{theo:convex_hull_B_n} or Theorem~\ref{theo:convex_hull_D_n}, $\Ker A$ a.s.\ has codimension $d$ in $\R^n$ and a.s.\ is in general position w.r.t.\ $\cA(B_n)$ or, respectively, $\cA(D_n)$.
\end{lemma}

These lemmas, combined with Lemma~\ref{lem:non-general_position}, imply that the value of $N$ does not change a.s.\ if we replace $\bar C$ by $C$ in the definition of $N$. Hence $N$ is a.s.\ a constant of the value given by Theorem~\ref{theo:regions_reflection_arrangement}, and then  Theorems~\ref{theo:convex_hull_A}, \ref{theo:convex_hull_B_n}, and \ref{theo:convex_hull_D_n} follow.


\subsection{General position: Proofs of Lemmas~\ref{lem:ker_A_gen_pos}, \ref{lem:ker_BD_gen_pos} and Proposition~\ref{prop:iid_sufficient_condition}}
\begin{proof}[Proof of Lemma~\ref{lem:ker_A_gen_pos}]
We use $\beta_1,\dots,\beta_n$ as coordinates on $\R^n$. Recall that $\Ker A$ is the set of solutions to the  system of effectively $d$ linear equations $\beta_1\xi_1+\dots + \beta_n \xi_n = 0$.
Define the linear function $T: \R^n \to L$ that maps $(\beta_1,\dots,\beta_n)$ to $(\beta_1-b,\dots,\beta_n - b)$, where $b= \frac 1n (\beta_1+\dots+\beta_n)$.
Note that $T(\Ker A) = L \cap \Ker A$ a.s.\ by $\xi_1+\dots+ \xi_n = 0$ a.s.\ Since $\Ker T = L^{\bot} =\{(a,\dots,a)\colon a\in\R\}$ belongs to $\Ker A$ a.s., the kernel of the restriction $T|_{\Ker A}$ also is $L^\bot$ a.s. Hence
$$
\dim (L\cap \Ker A) = \dim (\Ker A) - 1 = n-d-1 \;\; \text{a.s.},
$$
where the last equality holds since $\xi_1,\dots,\xi_d$ are linearly independent a.s., which in turn easily follows from the a.s.\ linear independence of $S_1, \ldots, S_d$ given by the general position assumption of Theorem~\ref{theo:convex_hull_A}. Therefore the rank of $A$ equals $d$ with probability $1$.

Let us prove that $L\cap \Ker A$ is in general position w.r.t.\ $\cA(A_{n-1})$. Take a linear subspace $K\subset \R^n$ of dimension $k$  that can be represented as the intersection of hyperplanes from $\cA(A_{n-1})$, i.e., hyperplanes of the form $\beta_ i = \beta_j$, $1\leq i < j\leq n$. According to~\eqref{1222}, we need to show that
$$
\dim (K\cap L \cap \Ker A) \stackrel{\text{a.s.}}{=}
\begin{cases}
k - d - 1, &\text{if } k \geq d+1,\\
0, &\text{if }  k \leq d+1.
\end{cases}
$$

Similarly to the above, we have $T( K \cap\Ker A) = K \cap L\cap\Ker A$ a.s.\ and the kernel of the restriction $T|_{K \cap\Ker A}$ a.s.\ is the one-dimensional linear space $L^\bot$. Thus it suffices to show that
$$
\dim (K\cap \Ker A) \stackrel{\text{a.s.}}{=}
\begin{cases}
k - d, &\text{if } k \geq d+1,\\
1, &\text{if }  k \leq d+1.
\end{cases}
$$

The linear subspace $K$ is given by a system of equations of the following type: the variables $\beta_1,\ldots, \beta_n$ are decomposed into $k$ non-empty groups and required to be equal inside each group. Since $(\xi_1,\dots,\xi_n)$ is an exchangeable tuple and we can apply a suitable transformation from the group $A_{n-1}$, it can be assumed without loss of generality that $K$ is given by the equations
\begin{align*}
\gamma_1:= \beta_1=\dots = \beta_{i_1},
\quad
\gamma_2:= \beta_{i_1+1}=\dots = \beta_{i_2},
\quad
\dots,
\quad
\gamma_k:= \beta_{i_{k-1} + 1} = \dots = \beta_{n},
\end{align*}
for some $1\leq i_1 < \dots < i_{k-1} < i_{k}: = n$.

Using $\gamma_1,\dots,\gamma_k$ as coordinates on $K$, observe that $K\cap \Ker A$ is given  by the equations
$$
\gamma_1 (\xi_1+\dots + \xi_{i_1}) + \gamma_2 (\xi_{i_1+1} + \dots + \xi_{i_2}) + \dots + \gamma_k (\xi_{i_{k-1}+1} + \dots + \xi_{n}) = 0,
$$
which imply that
$$
\gamma_1 S_{i_1} + \gamma_2 (S_{i_2} - S_{i_1}) + \dots + \gamma_k (0 - S_{i_{k-1}}) = 0 \text{ a.s.}
$$

Let $k= \dim K \geq d+1$. Then with probability $1$, the rank of this system of equations is maximal (namely, $d$) because $S_{i_1},\dots, S_{i_{d}}$ and hence, $S_{i_1}, S_{i_2} - S_{i_1}, \dots, S_{i_d} - S_{i_{d-1}}$, are linearly independent a.s.\ by our general position assumption. We have used that $i_d < i_k = n$. Then $K\cap \Ker A$, the space of solutions of the system, has dimension $k-d$ a.s.\ as required.

Let now $k< d+1$. Take some linear subspace $K_1 \supset K$ that can be represented as the intersection of hyperplanes from the arrangement $\mathcal A(A_{n-1})$ and satisfies $\dim K_1 = d+1$. Then apply the above to get $\dim (K_1\cap \Ker A) = d+1 - d = 1$ a.s. This yields $\dim (K\cap \Ker A)\leq 1$ a.s., but since $L^\bot \subset K\cap \Ker A$ a.s., we in fact have $\dim (K\cap \Ker A) = 1$ a.s., thus completing the proof.
\end{proof}

\begin{proof}[Proof of Lemma~\ref{lem:ker_BD_gen_pos}]
Consider the $B_n$ case first. By the general position assumption imposed in Theorem~\ref{theo:convex_hull_B_n}, the vectors $S_1,\dots,S_d$ are linearly independent a.s. Hence $\xi_1,\dots, \xi_d$ are linearly independent a.s.\ and the rank of the matrix $A$ equals $d$ with probability $1$. Then the codimension of $\Ker A$ equals $d$ a.s.

Letting $\beta_1,\dots,\beta_n$ denote the coordinates on $\R^n$,  observe that $\Ker A$ is given by $\beta_1\xi_1+\dots+ \beta_n\xi_n = 0$. To prove that $\Ker A$ a.s.\ is in general position w.r.t.\ $\cA(B_n)$, take a linear subspace $K\subset \R^n$ of dimension $k$ that can be represented as the intersection of hyperplanes from $\cA(B_n)$, that is, hyperplanes of the form
$$
\beta_ i = \beta_j \;\; (1\leq i < j\leq n),
\quad
\beta_ i = - \beta_j \;\; (1\leq i < j\leq n),
\quad
\beta_i = 0 \;\; (1\leq i\leq n).
$$
According to the definition of general position (see~\eqref{1222}), we have to show that
$$
\dim (K\cap \Ker A) \stackrel{\text{a.s.}}{=}
\begin{cases}
k - d, &\text{if } k \geq d,\\
0, &\text{if }  k \leq d.
\end{cases}
$$

The linear subspace $K$ is given by a system of equations of the following form. The variables $\beta_1, \dots, \beta_n$ are decomposed into $k+1$ distinguishable groups, all of which must be non-empty except the last one. All variables in the last group are required to be $0$. For the remaining variables, there is a unique choice of signs, which multiplies each variable by $+1$ or $-1$, such that the sign-changed variables are equal inside every group except the $(k+1)$st one.

Since the tuple $(\xi_1,\dots,\xi_n)$ is symmetrically exchangeable and we can apply a suitable transformation from the group $B_n$, it can be assumed without loss of generality that $K$ is given by the equations
\begin{align*}
&\gamma_1:= \beta_1=\dots = \beta_{i_1},
\quad
\gamma_2:= \beta_{i_1+1}=\dots = \beta_{i_2},
\quad
\dots,
\quad
\gamma_k:= \beta_{i_{k-1} + 1} = \dots = \beta_{i_k},\\
&\beta_{i_k+1} =\dots = \beta_n = 0,
\end{align*}
for some $1\leq i_1 < \dots < i_k \leq n$. We consider $\gamma_1,\dots,\gamma_k$ as coordinates on $K$. Then, $K\cap \Ker A$ is given (inside the linear space $K$) by the system of equations
$$
\gamma_1 (\xi_1+\dots + \xi_{i_1}) + \gamma_2 (\xi_{i_1+1} + \dots + \xi_{i_2}) + \dots + \gamma_k (\xi_{i_{k-1}+1} + \dots + \xi_{i_k}) = 0.
$$
Using the partial sums, this can be written as
$$
\gamma_1 S_{i_1} + \gamma_2 (S_{i_2} - S_{i_1}) + \dots + \gamma_k (S_{i_k} - S_{i_{k-1}}) = 0.
$$

Let first $k= \dim K \geq d$. Then with probability $1$, the rank of this system of equations is maximal (namely, $d$) since $S_{i_1},\dots, S_{i_d}$ and hence, $S_{i_1}, S_{i_2} - S_{i_1}, \dots, S_{i_d} - S_{i_{d-1}}$, are linearly independent a.s.\ by our general position assumption. Then $K\cap \Ker A$, the space of solutions of the system, has dimension $k-d$ a.s.\ as required.

Let now $k< d$. Take some linear subspace $K_1 \supset K$ representable as the intersection of hyperplanes from the arrangement $\mathcal A(B_n)$ and satisfying $\dim K_1 = d$. Applying the above to $K_1$, we get  $\dim (K_1\cap \Ker A) = d-d = 0$ a.s., which yields $\dim (K\cap \Ker A) = 0$ a.s., completing the proof in the $B_n$ case.

The $D_n$ case can be treated similarly, and we highlight only the main differences. Let  $K\subset \R^n$ be a linear subspace of dimension $k$ that can be represented as the intersection of hyperplanes of the form $\beta_ i = \pm \beta_j$,  $1\leq i < j\leq n$. Then $K$ has exactly the same form as in the $B_n$ case, but since the arrangement $\mathcal A(D_n)$ does not include the hyperplanes $\beta_i=0$, the last group of variables (that are required to be $0$) cannot contain exactly one element. Applying an appropriate transformation from the group $D_n$ allows to change only \textit{even} number of signs, therefore we can assume that $K$ either has the same form as in the $B_n$ case, or is given by
\begin{align*}
&\gamma_1:= \beta_1=\dots = \beta_{i_1},
\quad
\dots,
\quad
\gamma_{k-1}:= \beta_{i_{k-2}+1}=\dots = \beta_{i_{k-1}},
\\
&\gamma_k:= \beta_{i_{k-1} + 1} = \dots =\beta_{n-1} = -\beta_{n},
\end{align*}
for some $1\leq i_1 < \dots < i_{k-1} < i_k := n$. In the former case, the same argument as in the $B_n$ case applies. In the latter case,
$K\cap \Ker A$ is given (inside the linear space $K$) by
$$
\gamma_1 S_{i_1} + \gamma_2 (S_{i_2} - S_{i_1}) + \dots + \gamma_k (S_{n}^* - S_{i_{k-1}}) = 0,
$$
where we recall that $S_n^* = S_{n-1} -\xi_n$. From now on, we can apply the same argument as in the $B_n$ case, but  with $S_n$ replaced by $S_n^*$.
\end{proof}

\begin{proof}[Proof of Proposition~\ref{prop:iid_sufficient_condition} and Theorem~\ref{theo:intro}]
It suffices to prove the equivalence of (i), (ii) and (iii) in Proposition~\ref{prop:iid_sufficient_condition}, because then Theorem~\ref{theo:intro} follows as a particular case of Theorem~\ref{theo:convex_hull_B_n}.

\vspace*{2mm}
\noindent
\textit{Proof of (i) $\Rightarrow$ (ii).}
Assume by contraposition that $\delta := \P[\xi_1 \in H] >0$ for some affine hyperplane $H= H_0 + v$, where $H_0$ is a hyperplane passing through the origin. Since the distribution of $\xi_1$ is symmetric, we have $\P[S_2\in H_0] \geq \delta^2 > 0$ and hence,
$$
\P[S_{2}\in H_0, S_4\in H_0, \ldots,S_{2d} \in H_0]\geq \delta^{2d} >0,
$$
a contradiction to (i).

\vspace*{2mm}
\noindent
\textit{Proof of (ii) $\Rightarrow$ (iii).}
Let $H_0$ be given by the equation $f(x)=0$, where $f:\R^d\to\R$ is a linear functional. By (ii), we have $\P[S_1\in H_0] =0$. Let $i\geq 2$. Using the identity $f(S_i) = f(S_{i-1}) + f(\xi_i)$, we obtain
$$
\P[S_i \in H_0] = \P[f(S_i)=0]
=
\int_{\R} \P[f(\xi_i) = -y]\, \P[f(S_{i-1}) \in {\rm d} y]
=
0
$$
since by (ii), $\P[f(\xi_i) = -y]=0$ for all $y\in\R$.

\vspace*{2mm}
\noindent
\textit{Proof of (iii) $\Rightarrow$ (i).}
The vectors $S_{i_1},\ldots,S_{i_d}$ are linearly dependent if and only if $S_{i_{k}}$ can be linearly expressed through $S_{i_1},\ldots, S_{i_{k-1}}$ for some $2\leq k \leq d$, or if $S_{i_1}=0$. The latter event has probability zero by (iii). To prove that the former event also has probability $0$, denote by $\lspan(y_1,\ldots,y_{k-1})$ the linear subspace spanned by vectors $y_1,\ldots,y_{k-1}\in\R^d$. Then
\begin{multline*}
\P[S_{i_k} \in \lspan(S_{i_1},\ldots,S_{i_{k-1}})]\\
\begin{aligned}
&=
\P[S_{i_k} - S_{i_{k-1}} \in \lspan(S_{i_1},\ldots,S_{i_{k-1}})] \\
&=
\int_{\R^{k-1}}  \P[S_{i_k} - S_{i_{k-1}} \in \lspan(y_1,\ldots,y_{k-1})]\,  \P [S_{i_1}\in {\rm d}y_1,\ldots,S_{i_{k-1}} \in {\rm d}y_{k-1}] \\
&=
\int_{\R^{k-1}}  \P[S_{i_k-i_{k-1}} \in \lspan(y_1,\ldots,y_{k-1})] \, \P [S_{i_1}\in {\rm d}y_1,\ldots,S_{i_{k-1}} \in {\rm d}y_{k-1}] \\
&=
0
\end{aligned}
\end{multline*}
since the integrand is $0$ by (iii). Hence the probability that $S_{i_1},\ldots,S_{i_d}$ are linearly dependent is $0$.
\end{proof}

\subsection{Non-general position: Proofs of Lemma~\ref{lem:non-general_position} and Proposition~\ref{prop:non_general_absorbtion_probab}}\label{subsec:proof_without_density}
\begin{proof}[Proof of Lemma~\ref{lem:non-general_position}]
Let us prove~\eqref{eq:R=bar_R}, that is
\begin{equation} \label{eq:R=bar_R_repeat}
\{R\in \cR(\cA)\colon \bar R\cap L_{n-d}\ne \{0 \}\} = \{R\in \cR(\cA)\colon R\cap L_{n-d}\ne\varnothing\}.
\end{equation}
Since $0\notin R$, the assumption $R\cap L_{n-d} \ne \varnothing$ clearly implies that $\bar R\cap L_{n-d}\ne \{0 \}$.
To prove the other inclusion in~\eqref{eq:R=bar_R_repeat}, we assume, by contraposition, that $\bar R\cap L_{n-d}\ne \{0\}$ but $R\cap L_{n-d} = \varnothing$. Note that $\bar R$ is a closed polyhedral cone, that is an intersection of finitely many closed half-spaces whose boundaries are hyperplanes passing through the origin.

Given some face $F$ of $\bar R$, we denote by $\lspan F$ its linear hull and by $m \in \{0, \ldots, n\}$ the dimension of $\lspan F$ (we assume that $\bar R $ itself  is a face).  The relative interior $\relint F$ of the face $F$ is defined as the interior of $F$ taken w.r.t.\ the linear hull $\lspan F$ as the ambient space. It is known that $\bar R$ is the disjoint union of the relative interiors of its faces.  Hence, there is a face $F\neq \{0\}$ of $\bar R$ such that $L_{n-d}\cap \relint F \neq \varnothing$. Since $\bar R$ is a cone, the dimension of $F$ is at least one. Note that $F\neq \bar R$ by the assumption that $R\cap L_{n-d} = \varnothing$, hence $\dim F = m \notin \{0,n\}$. Since $L_{n-d}$ is in general position w.r.t.\ $\cA$, the dimension of the linear space $V_0 := L_{n-d}\cap \lspan F$ is $\max (m-d, 0)$. In fact, it equals $m - d \neq 0$, because otherwise we would have $L_{n-d}\cap \lspan F = \{0\}$, which is a contradiction.  Thus, we can construct linear spaces $V_1$ and $V_2$ such that $\lspan F = V_0 \oplus V_1$, $L_{n-d} = V_0\oplus V_2$, and $V_0 \oplus V_1 \oplus V_2 = \R^n$. We then have $\dim V_1 = d$ and $\dim V_2 = n-m\neq 0$. Note that $V_1$ and $V_2$ can be taken so that $V_0 \bot V_1$ and $V_0 \bot V_2$ but we do not necessarily have $V_1 \bot V_2$.

Take some $x\in L_{n-d}\cap \relint F$.
 The support (or tangent) cone  of $\bar R$ at its face $F$ is defined as
$$
A(F)
=
\pos(\bar R - x)
=
\{y\in\R^n\colon \exists \eps >0 \text{ such that } x+ \eps y \in \bar R\},
$$
where $\pos (M)$, the positive hull of $M$,  is the minimal convex cone containing the set $M$.
It is known that $A(F)$ is a closed convex cone containing $\lspan F$ and not depending on the choice of $x\in \relint F$. This assumption on $x$ also ensures that there is a $\delta>0$ such that 
\begin{equation} \label{eq: R and A(F)}
B_\delta(x) \cap \bar R = B_\delta(x) \cap (x+ A(F)).
\end{equation}

Let $z\in R$ be some point from the interior of $\bar R$. Write $z -x = v_0+v_1+v_2\in A(F)$ with $v_i\in V_i$, $i=0,1,2$. Since $-(v_0 +v_1)\in  \lspan F \subset A(F)$, we have $v_2\in A(F)$. In fact, the same argument applies to any point in a sufficiently small ball around $z$. Observe that $v_2$ is the projection of $z-x$ onto $V_2$ along $V_0\oplus V_1$, and that the projection of the ball around $z-x$ covers some set of the form $B_{r'}(v_2)\cap V_2$, where  $B_{r'}(v_2)$ is the ball of radius $r'>0$ around $v_2$. This proves that $B_{r'}(v_2) \cap V_2 \subset A(F)$, but since $V_0\oplus V_1 = \lspan F \subset A(F)$, we even have $B_r(v_2) \subset A(F)$ for some $0< r \leq r'$ by the convex cone property of $A(F)$. Then, by the same property, $B_{\eps r} (\eps v_2) \subset B_\delta (0) \cap A(F)$ for all sufficiently small $\eps>0$. Hence, using \eqref{eq: R and A(F)}, we get $B_{\eps r} (x+\eps v_2) \subset B_\delta (x) \cap (x +A(F)) \subset \bar R$. This implies that $x + \eps v_2$ is in the interior of $\bar R$, which is a contradiction because we also have $x\in L_{n-d}$,  $v_2 \in L_{n-d}$ and, consequently, $x+\eps v_2 \in L_{n-d}$.


Now we prove~\eqref{eq:lem:non-general_position1}; the proof of~\eqref{eq:lem:non-general_position2} is analogous. So, we need to prove that
$$
\#\{R\in \cR(\cA)\colon \bar R\cap L_{n-d}'\neq \{0\}\} \geq \#\{R\in \cR(\cA)\colon \bar R\cap L_{n-d}\neq  \{0\}\}.
$$
Let $\mathbf{Gr}(n-d,n)$ be the Grassmannian of all $(n-d)$-dimensional linear subspaces in $\R^n$ endowed with the following metric: the distance between two linear subspaces $M$ and $N$ is defined as the operator norm of the difference of the orthogonal projections in $\R^n$ onto $M$ and $N$.
This  metric coincides with the Hausdorff distance between the sets obtained by intersecting $M$ and $N$ with the unit ball in $\R^n$; see~\citet[Section 39]{AkhierzerGlazman}.
Hence the Grassmannian $\mathbf{Gr}(n-d,n)$ is a compact metric space. There is a unique probability measure on it (the Haar measure) invariant under rotations of $\R^n$.

The set of subspaces that are in general position w.r.t.\ $\cA$ is dense in $\mathbf{Gr}(n-d,n)$. Indeed,  the complement of this set has zero Haar measure by \cite[Lemma~13.2.1]{schneider_weil_book}, and  the Haar measure of any ball in $\mathbf{Gr}(n-d,n)$ is strictly positive, which is a consequence of the compactness of $\mathbf{Gr}(n-d,n)$ and the transitivity of the action of the orthogonal group on $\mathbf{Gr}(n-d,n)$. For any chamber $R\in \cR(\cA)$, the set
$$
\{M\in\mathbf{Gr}(n-d,n)\colon \bar R \cap M= \{0\}\}
$$
is open in $\mathbf{Gr}(n-d,n)$. Therefore, there exists a neighborhood $U$ of $L_{n-d}'$ such that for all linear subspaces $M\in U$ we have
$$
\{R\in \cR(\cA)\colon \bar R\cap M = \{0\}\} \supset \{R\in \cR(\cA)\colon \bar R\cap L_{n-d}'  = \{0\}\}.
$$
We finish the proof by taking $M \in U$ that is in general position w.r.t.\ $\cA$ and noting that by \eqref{eq:R=bar_R} and Theorem~\ref{1229}, it holds
$$
\#\{R\in \cR(\cA)\colon \bar R\cap M = \{0\}\}=\#\{R\in \cR(\cA)\colon \bar R\cap L_{n-d} = \{0\}\}.
$$
\end{proof}

\begin{proof}[Proof of Proposition~\ref{prop:non_general_absorbtion_probab}]
For concreteness, we consider the case $B_n$ and prove only the inequality
$$\P[0 \in H_{n,d}] \leq  \P[0 \in H_{n,d}'].$$
Recall that $(\xi_1',\dots,\xi_n')$ is a symmetrically exchangeable tuple of random vectors in $\R^d$ and $H_{n,d}'=\conv(S_1', \dots, S_n')$ is the convex hull of the partial sums $S_k'=\xi_1'+\dots+\xi_k'$. Let $A'$ be the $d\times n$-matrix with columns $\xi_1',\dots,\xi_n'$. By Lemma~\ref{lem:ABD},
$$
\P[0\in H_{n,d}']
=\frac {\E N'}{2^n n!},
$$
where $N'$ is the numbers of closed Weyl chambers of type $B_n$ (non-trivially) intersected by $\Ker A'$:
$$
N' = \sum_{g\in B_n} \ind_{\{\Ker A' \cap  (g\bar C) \neq \{0\}\}}.
$$

We imposed no general position assumption on $(\xi_1',\dots,\xi_n')$ and we cannot claim that $\Ker A'$ is in general position w.r.t.\ the arrangement $\cA(B_n)$. In particular, the random variable $N'$ need not be a constant a.s. Moreover, we do not even known the exact codimension of $\Ker A'$, but we can claim that it is at most $\min(d,n)$. Let $F\subset \Ker A'$ be any random linear subspace of $\R^n$ of a.s.\ codimension $d$. For example, we may define it as follows. Put $\kappa= \min(d-\codim (\Ker A'), n)$ and take $F = \bigcap_{i=1}^\kappa X_i^\perp \cap \Ker A'$, where $X_1,\dots,X_n$ are i.i.d.\ random vectors that are distributed on $\S^{n-1}$ and independent of $A'$.

Clearly,
$$
\P[0\in H_{n,d}']
\geq \frac {\E \tilde N}{2^n n!}, \mbox{ where }
\tilde N:= \sum_{g\in B_n} \ind_{\{F \cap  (g\bar C) \neq \{0\}\}}.
$$
By Lemma~\ref{lem:non-general_position}, we have $\tilde N\geq N$ a.s.\ for $N$ defined by \eqref{eq:N=} with $G=B_n$. By \eqref{eq:reduction}, the claim follows.
\end{proof}

\section{Open questions}
Our results, except the estimates of Theorem~\ref{theo:simple_RW_estimate} and Proposition~\ref{prop:non_general_absorbtion_probab}, do not apply to simple random walks on the lattice $\Z^d$. The next problem does not seem trivial even for $d=2$.
\begin{problem}
Let $S_1,\dots, S_n$ be a simple random walk on $\Z^d$ starting at the origin. Compute exactly the probability that
$\conv(S_1,\dots, S_n)$ contains the origin.
Compute exactly the conditional probability that $\conv (S_1,\dots,S_{n-1})$ contains the origin given that $S_n=0$.
\end{problem}

\begin{problem}
Prove analogues of Theorems~\ref{theo:CLT} and~\ref{theo:LD} for the simple random walks (and bridges) on $\Z^d$.
\end{problem}

The answer to the next question should be non distribution-free (for $x\neq 0$) and seems to be unknown even in the case of standard normal increments.
\begin{problem}
Let $S_1,\dots, S_n$ be a random walk in $\R^d$ starting at the origin. Compute the probability that
$\conv(S_1,\dots, S_n)$ contains a given point $x\in\R^d$.
\end{problem}

The same question makes sense for a Brownian motion.
\begin{problem}
Let $\{B(t)\colon t\geq 0\}$ be a standard Brownian motion in $\R^d$ starting at the origin. Compute the probability that $\conv\{B(t) \colon 0\leq t\leq 1\}$ contains a given point $x\in\R^d$.
\end{problem}


\section*{Acknowledgement}

We thank the anonymous referees for their valuable comments and suggestions.

\bibliographystyle{plainnat}
\bibliography{conv_hull_weyl_chambers}

\begin{thebibliography}{34}
\providecommand{\natexlab}[1]{#1}
\providecommand{\url}[1]{\texttt{#1}}
\expandafter\ifx\csname urlstyle\endcsname\relax
  \providecommand{\doi}[1]{doi: #1}\else
  \providecommand{\doi}{doi: \begingroup \urlstyle{rm}\Url}\fi

\bibitem[Akhiezer and Glazman(1981)]{AkhierzerGlazman}
N.~I. Akhiezer and I.~M. Glazman.
\newblock \emph{Theory of linear operators in {H}ilbert space. {V}ol. {I}}.
\newblock Pitman (Advanced Publishing Program), Boston, Mass.-London, 1981.

\bibitem[Allendoerfer(1948)]{cA48}
C.~Allendoerfer.
\newblock {S}teiner's formulae on a general {$S^{n+1}$}.
\newblock \emph{Bull. Amer. Math. Soc.}, 54:\penalty0 128--135, 1948.

\bibitem[Amelunxen and Lotz()]{AL15}
D.~Amelunxen and M.~Lotz.
\newblock Intrinsic volumes of polyhedral cones: a combinatorial perspective.
\newblock \emph{Preprint}.
\newblock Available at https://arxiv.org/abs/1512.06033.

\bibitem[Amelunxen et~al.(2014)Amelunxen, Lotz, McCoy, and Tropp]{ALMT14}
D.~Amelunxen, M.~Lotz, M.~McCoy, and J.~Tropp.
\newblock Living on the edge: Phase transitions in convex programs with random
  data.
\newblock \emph{Inform. Inference}, 3:\penalty0 224--294, 2014.

\bibitem[{Drton} and {Klivans}(2010)]{drton_klivans}
M.~{Drton} and C.~J. {Klivans}.
\newblock {A geometric interpretation of the characteristic polynomial of
  reflection arrangements.}
\newblock \emph{{Proc. Am. Math. Soc.}}, 138\penalty0 (8):\penalty0 2873--2887,
  2010.

\bibitem[Eldan(2014)]{eldan_extremal}
R.~Eldan.
\newblock Extremal points of high-dimensional random walks and mixing times of
  a {B}rownian motion on the sphere.
\newblock \emph{Ann. Inst. Henri Poincar\'e Probab. Stat.}, 50\penalty0
  (1):\penalty0 95--110, 2014.

\bibitem[F{\'e}ray et~al.(2016)F{\'e}ray, M{\'e}liot, and
  Nikeghbali]{feray_meliot_nikeghbali}
V.~F{\'e}ray, P.-L. M{\'e}liot, and A.~Nikeghbali.
\newblock \emph{Mod-{$\varphi$} convergence: Normality zones and precise
  deviations}.
\newblock Springer, Cham, 2016.

\bibitem[Flajolet and Sedgewick(2009)]{Flajolet_book}
P.~Flajolet and R.~Sedgewick.
\newblock \emph{Analytic combinatorics}.
\newblock Cambridge University Press, Cambridge, 2009.

\bibitem[Gao et~al.(2003)Gao, Hug, and Schneider]{GHS03}
F.~Gao, D.~Hug, and R.~Schneider.
\newblock Intrinsic volumes and polar sets in spherical space.
\newblock \emph{Math. Notae}, 41:\penalty0 159--176, 2003.

\bibitem[Goldstein et~al.(2017)Goldstein, Nourdin, and Peccati]{GNP14}
L.~Goldstein, I.~Nourdin, and G.~Peccati.
\newblock Gaussian phase transitions and conic intrinsic volumes: Steining the
  {S}teiner formula.
\newblock \emph{Ann. Appl. Probab.}, 27\penalty0 (1):\penalty0 1--47, 2017.

\bibitem[Grove and Benson(1985)]{benson_grove}
L.~C. Grove and C.~T. Benson.
\newblock \emph{{Finite reflection groups. 2nd ed.}}
\newblock 1985.

\bibitem[Herglotz(1943)]{gH43}
G.~Herglotz.
\newblock {\"U}ber die {S}teinersche {F}ormel f{\"u}r {P}arallelfl{\"a}chen.
\newblock \emph{Abh. Math. Sem. Hansischen Univ.}, 15:\penalty0 165--177, 1943.

\bibitem[Hug and Schneider(2016)]{HS15}
D.~Hug and R.~Schneider.
\newblock Random conical tessellations.
\newblock \emph{St. Petersburg Math. J.}, 56\penalty0 (2):\penalty0 395--426,
  2016.

\bibitem[{Klivans} and {Swartz}(2011)]{klivans_swartz}
C.~J. {Klivans} and E.~{Swartz}.
\newblock {Projection volumes of hyperplane arrangements.}
\newblock \emph{{Discrete Comput. Geom.}}, 46\penalty0 (3):\penalty0 417--426,
  2011.

\bibitem[Kowalski and Nikeghbali(2010)]{kowalski_nikeghbali}
E.~Kowalski and A.~Nikeghbali.
\newblock Mod-{P}oisson convergence in probability and number theory.
\newblock \emph{Int. Math. Res. Not. IMRN}, \penalty0 (18):\penalty0
  3549--3587, 2010.

\bibitem[McCoy and Tropp(2014)]{MT14}
M.~McCoy and J.~Tropp.
\newblock From {S}teiner formulas for cones to concentration of intrinsic
  volumes.
\newblock \emph{Discrete Comput. Geom.}, 51:\penalty0 926--963, 2014.

\bibitem[Nevzorov(2000)]{nevzorov_book}
V.~B. Nevzorov.
\newblock \emph{Records: {M}athematical theory}.
\newblock Providence: AMS, 2000.
\newblock Transl. from the Russian.

\bibitem[Orlik and Terao(1992)]{OT92}
P.~Orlik and H.~Terao.
\newblock \emph{Arrangements of hyperplanes}.
\newblock Probability and its Applications. Springer-Verlag, Berlin, 1992.

\bibitem[Santal{\'o}(1950)]{lS50}
L.~Santal{\'o}.
\newblock On parallel hypersurfaces in the elliptic and hyperbolic
  {$n$}-dimensional space.
\newblock \emph{Proc. Amer. Math. Soc.}, 1:\penalty0 325--330, 1950.

\bibitem[Santal\'o(1976)]{lS76}
L.~Santal\'o.
\newblock \emph{Integral geometry and geometric probability}.
\newblock Addison-Wesley Publishing Company, Reading, 1976.

\bibitem[Schl{\"a}fli(1950)]{Schl50}
L.~Schl{\"a}fli.
\newblock Theorie der vielfachen {K}ontinuit{\"a}t.
\newblock In \emph{Gesammelte Mathematische Abhandlungen}, pages 167--387.
  Springer, 1950.

\bibitem[Schneider()]{Schneider16}
R.~Schneider.
\newblock Combinatorial identities for polyhedral cones.
\newblock \emph{Accepted to St. Petersburg Math. J.}
\newblock Available at
  http://home.mathematik.uni-freiburg.de/rschnei/Comb.Ident.rev.pdf.

\bibitem[Schneider and Weil(2008)]{schneider_weil_book}
R.~Schneider and W.~Weil.
\newblock \emph{Stochastic and integral geometry}.
\newblock Probability and its Applications. Springer--Verlag, Berlin, 2008.

\bibitem[Sloane()]{sloane}
N.~J.~A. Sloane.
\newblock The on-line encyclopedia of integer sequences.
\newblock http://www.research.att.com/~njas/sequences/.

\bibitem[Sparre~Andersen(1949)]{sparre_andersen0}
E.~Sparre~Andersen.
\newblock On the number of positive sums of random variables.
\newblock \emph{Skand. Aktuarietidskr.}, 32:\penalty0 27--36, 1949.

\bibitem[Sparre~Andersen(1953)]{Sparre1953}
E.~Sparre~Andersen.
\newblock On the fluctuations of sums of random variables.
\newblock \emph{Math. Scand.}, 1:\penalty0 263--285, 1953.

\bibitem[{Stanley}(2007)]{stanley_book}
R.~P. {Stanley}.
\newblock {An introduction to hyperplane arrangements.}
\newblock In \emph{{Geometric combinatorics}}, pages 389--496. 2007.

\bibitem[Steiner(1826)]{jS26}
J.~Steiner.
\newblock Einige {G}esetze {\"u}ber die {T}heilung der {E}bene und des
  {R}aumes.
\newblock \emph{Journal f{\"u}r die reine und angewandte Mathematik},
  1:\penalty0 349--364, 1826.

\bibitem[Suter(2000)]{suter}
R.~Suter.
\newblock Two analogues of a classical sequence.
\newblock \emph{J. Integer Seq.}, 3\penalty0 (1):\penalty0 Article 00.1.8,
  2000.

\bibitem[Tikhomirov and Youssef(2017)]{tikhomirov_youssef2}
K.~Tikhomirov and P.~Youssef.
\newblock When does a discrete-time random walk in {$\mathbb R^n$} absorb the
  origin into its convex hull?
\newblock \emph{Ann. Probab.}, 45\penalty0 (2):\penalty0 965--1002, 2017.

\bibitem[Vysotsky and Zaporozhets(2017)]{vysotsky_zaporozhets}
V.~Vysotsky and D.~Zaporozhets.
\newblock Convex hulls of multidimensional random walks.
\newblock \emph{Accepted to Trans. Amer. Math. Soc.}, 2017.
\newblock Available at http://arxiv.org/abs/1506.07827.

\bibitem[Wendel(1962)]{wendel}
J.~G. Wendel.
\newblock A problem in geometric probability.
\newblock \emph{Math. Scand.}, 11:\penalty0 109--111, 1962.

\bibitem[Wilf(1993)]{wilf}
H.~S. Wilf.
\newblock The asymptotic behavior of the {S}tirling numbers of the first kind.
\newblock \emph{{J. Comb. Theory, Ser. A}}, 64\penalty0 (2):\penalty0 344--349,
  1993.

\bibitem[Zaslavsky(1975)]{tZ75}
T.~Zaslavsky.
\newblock Facing up to arrangements: face-count formulas for partitions of
  space by hyperplanes.
\newblock \emph{Mem. Amer. Math. Soc.}, 1\penalty0 (issue 1, 154), 1975.

\end{thebibliography}

\end{document}